\documentclass[journal]{IEEEtran}

\usepackage{graphicx} 
\usepackage{amsmath, amssymb, amsfonts}
\usepackage{color}
\usepackage{subfig}
\usepackage{epstopdf}
\usepackage{setspace}
\usepackage{multirow}

\newcommand{\eome}{{\rm \nabla\kern-.6em \nabla}}
\newcommand{\diag}{\mathrm{diag}}
\newtheorem{definition}{Definition}
\newtheorem{theo}{Theorem}
\newtheorem{lemma}{Lemma}
\newtheorem{coro}{Corollary}
\newtheorem{remarkTh}{Remark}

\usepackage[bb=boondox]{mathalfa}
\newcommand{\Chi}{\mathfrak{X}}

\newcommand{\supeq}{\geqslant}
\newcommand{\infeq}{\leqslant}
\newcommand{\V}{\mathcal{V}}

\newcommand{\R}{\mathbb{R}}

\newenvironment{proof}[1][Proof]{\textbf{#1.} }{\hfill\rule{0.5em}{0.5em}}
\newenvironment{remark}[0]{\begin{remarkTh}}{\hfill$\square$\end{remarkTh}}

\DeclareMathOperator\He{He}
\DeclareMathOperator\sign{sign}

\renewcommand\epsilon{\varepsilon}

\title{\LARGE \bf Practical Stability Analysis of a Drilling Pipe under Friction with a PI-Controller}

\author{Matthieu~Barreau, Fr\'ed\'eric~Gouaisbaut and Alexandre~Seuret
\thanks{M. Barreau, A. Seuret, F. Gouaisbaut are with LAAS - CNRS, Universit\'e de Toulouse, CNRS, UPS, France e-mail: (mbarreau,aseuret,fgouaisb@laas.fr). This work is supported by the ANR project SCIDiS contract number 15-CE23-0014.}
}

\begin{document}

\maketitle
\thispagestyle{empty}
\pagestyle{empty}

\begin{abstract}
This paper deals with the stability analysis of a drilling pipe controlled by a PI controller. 
The model is a coupled ODE / PDE and is consequently of infinite dimension. Using recent advances in time-delay systems, we derive a new Lyapunov functional based on a state extension made up of projections of the Riemann coordinates. 
First, we will provide an exponential stability result expressed using the LMI framework. This result is dedicated to a linear version of the torsional dynamic. On a second hand, the influence of the nonlinear friction force, which may generate the well-known stick-slip phenomenon, is analyzed through a new stability theorem.
Numerical simulations show the effectiveness of the method and that the stick-slip oscillations cannot be weaken using a PI controller.
\end{abstract}

\section{Introduction}

Studying the behavior of complex machineries is a real challenge since they usually present nonlinear and coupled behaviors \cite{weaver1990vibration}. A drilling mechanism is a very good example of this. Many nonlinear effects can occur on the drilling pipe such as bit bouncing, stick-slip or whirling \cite{challamel2000rock}. These phenomena induce generally some vibrations, increasing the drillpipe fatigue and affecting therefore the life expectancy of the well. The first challenge is then to provide a dynamical model which reflects such behaviors.

Looking at the literature, many models exist from the simplest finite-dimensional ones presented in for instance \cite{canudasdewit:hal-00394990,hinfty} to the more complex but more realistic infinite-dimensional systems. Finite-dimensional systems were an important first step since they showed which characteristics are responsible of vibrations in the well. Nevertheless, they are too far from the physical laws which are expressed in terms of partial differential equation. Then, a coupled finite/infinite dimensional model seems more natural in the context of a drilling pipe and it was proposed in \cite{challamel2000rock,fridman2010bounds,marquez2015analysis} for instance. 

The second challenge was then to design a controller to remove, or at least weaken, these undesirable effects. Many control techniques were applied on the finite-dimensional model from the simple PI controller investigated in \cite{canudasdewit:hal-00394990,christoforou2003fully}, to more advance controllers as sliding mode control \cite{navarro2009} or $H_{\infty}$ \cite{hinfty}. Nevertheless, extending these controllers on the coupled finite/infinite dimensional system is not straightforward.

The last decade has seen many developments regarding the analysis of infinite-dimensional systems. The semigroup theory, investigated in \cite{tucsnak2009observation} for instance, was a great tool to simplify the proof of existence and uniqueness of a solution to this kind of problems. That leads to an extension of the Lyapunov theory to some classes of Partial Differential Equations (PDE) \cite{bastin2016stability,coron2007control,prieur2016wave}. These advances have given rise to the stability analysis of the linearized infinite-dimensional drilling pipe with a PI controller \cite{ATJdrilling,ATJECC}. Since this kind of controller provides only two degree of freedom, to better enhance the performances, slightly different controllers arose. One of the most famous is the modified PI controller in \cite{tucker1999effective}, but there is also a delayed PI or a flatness based control in \cite{marquez2015analysis}. More complex controllers, coming from the backstepping technique for PDE, originally developed in \cite{krstic2009delay}, were also applied in \cite{basturk2017observer,bresch2014output,roman2016backstepping} for instance.

Nevertheless, these techniques almost always use a Lyapunov argument to conclude and they consequently suffer from the lack of an efficient Lyapunov functional for coupled systems. Recent advances in the domain of time-delay systems \cite{seuret:hal-01065142} have lead to a hierarchy of Lyapunov functionals which are very efficient for coupled Ordinary Differential Equation (ODE) / string equation \cite{besselString}. Since it relies on a state extension, the stability analysis cannot be assessed manually but it translates into an optimization problem expressed using Linear Matrix Inequalities (LMIs) and consequently easily solvable.

This paper takes advantage of this enriched Lyapunov functional to revisit the stability analysis of a PI controlled infinite-dimensional model of a drilling pipe for the torsion only. The first contribution of this article results in Theorem~1, which provides a Linear Matrix Inequality (LMI) to ensure the asymptotic stability of the linear closed-loop system. The second theorem deals with the practical stability of the controlled nonlinear plant. It shows for example that if the linear system is stable, the nonlinear system is also stable. Moreover, it provides an accurate bound on the oscillations during the stick-slip.

The article is organized as follows. Section~2 discusses the different models presented in the literature and enlighten the importance of treating the infinite-dimensional problem. Section~3 is the problem statement. Section~4 is dedicated to the study of the linear system. This is a first step before dealing with the nonlinear system, which is the purpose of Section~5. Section~6 finally proposes simulations to demonstrate the effectiveness of this approach and conclude about the design of a PI controller.

\textbf{Notations:} For a multi-variable function $(x,t) \mapsto u(x,t)$, the notation $u_t$ stands for $\frac{\partial u}{\partial t}$ and $u_x = \frac{\partial u}{\partial x}$. We also use the notations $L^2 = L^2((0, 1); \mathbb{R})$ and for the Sobolov spaces: $H^n = \{ z \in L^2; \forall m \infeq n, \frac{\partial^m z}{\partial x^m} \in L^2 \}$. 
The norm in $L^2$ is $\|z\|^2 =  \int_{\Omega} |z(x)|^2  dx = \left<z,z\right>$. 
For any square matrices $A$ and $B$, the following operations are defined: $\text{He}(A) = A + A^{\top}$ and $\text{diag}(A,B) = \left[ \begin{smallmatrix}A & 0\\ 0 & B \end{smallmatrix} \right]$.
The set of positive definite matrices of size $n$ is denoted by $\mathbb{S}^n_+$ and, for simplicity, a matrix $P$ belongs to this set if $P \succ 0$.

\section{Model Description}
\label{sec:model}

A drilling pipe is a mechanism used to pump oil deep under the surface thanks to a drilling pipe as illustrated in Figure~\ref{fig:drilling}. Throughout the thesis, $\Phi(\cdot,t)$ is the twisting angle along the pipe and then $\Phi(0, t)$ and $\Phi(L, t)$ are the angles at the top and at the bottom of the well respectively. The well is a long metal rode of around one kilometer and consequently, the rotational velocity applied at the top using the torque $u_1(t)$ is different from the one at the bottom. Moreover, the interaction of the bit with the rock at the bottom are modeled by torque $T$, which depends on $\Phi_t(L, t)$. 

As the bit drills the rock, an axial compression of the rod occurs and is denoted $\Psi$.  This compression arises because of the propagation along the rode of the vertical force $u_2$ applied at the top to push up and down the well.

This description leads naturally to two control objectives to prevent the mechanism from breaking. The first one is to maintain the rotational speed at the end of the pipe $\Phi(L, t)$ at a constant value denoted here $\Omega_0$, preventing any twisting of the pipe. The other one is to keep the penetration rate constant such that there is no compression along the rode.

Several models have been proposed in the literature to achieve these control objectives. They are of very different natures and lead to a large variety of analysis and control techniques. The book \cite{marquez2015analysis} (chap. 2) and the survey \cite{saldivar:hal-01425845} provide overviews of these techniques, which are, basically, of four kinds. To better motivate the model used in the sequel, a brief overview of the existing modeling tools is proposed but the reader can refer to \cite{saldivar:hal-01425845} and the original papers to get a better understanding of how the models are constructed.

\begin{figure}
	\centering
	\includegraphics[width=7.5cm]{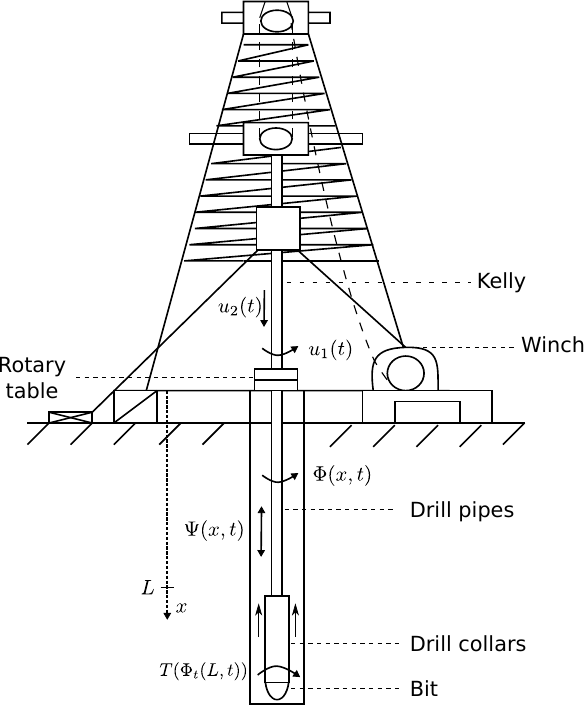}
	\caption{Schematic of a drilling mechanism originally taken from \cite{saldivar2016control}. Data corresponding to physical values are given in Table~\ref{tab:parameters}.}
	\label{fig:drilling}
\end{figure}

\subsection{Lumped Parameter Models (LPM)}

These models are the first obtained in the literature \cite{canudasdewit:hal-00394990,Liu2014,hinfty} and the full mechanism is described by a sequence of harmonic oscillators. They can be classified into two main categories:
\begin{enumerate}
	\item The first kind assumes that the dynamics of the twisting angles $\Phi(0, t) = \Phi_r(t)$ (at the top) and $\Phi(L, t) = \Phi_b(t)$ (at the bottom) are described by two coupled harmonic oscillators. The torque $u_1$ driving the system is applied on the dynamic of $\Phi_r$ and the controlled angle is $\Phi_b$. The axial dynamic is not taken into account in this model. This model can be found in \cite{canudasdewit:hal-00394990,1387580,hinfty} for instance.
	\item The other two degrees of freedom model is described in \cite{Liu2014,78eaa3b25c4343b99eba5cf8f29f04c1} for example. There also are two coupled harmonic oscillators for $\Psi(L,t)$ and $\Phi(L, t)$ representing the axial and torsional dynamics. This model only considers the motions at the end of the pipe and forget about the physics occurring along the rode.
\end{enumerate}
The first class of models can be described by the following set of equations:
\begin{equation} \label{eq:lumpedMassModel}
	\left\{
		\begin{array}{l}
			I_r \ddot{\Phi}_r + \lambda_r (\dot{\Phi}_r - \dot{\Phi}_b) + k (\Phi_r - \Phi_b) + d_r \dot{\Phi}_r = u_1, \\
			I_b \ddot{\Phi}_b + \lambda_b (\dot{\Phi}_b - \dot{\Phi}_r) + k (\Phi_b - \Phi_r) + d_b \dot{\Phi}_b = -T(\dot{\Phi}_b),
		\end{array}
	\right.
\end{equation}
where the parameters are given in Table~\ref{tab:lumpedMassModel}. $T$ is a torque modeled by a nonlinear function of $\dot{\Phi}_b$, it describes the bit-rock interaction\footnote{See \cite{marquez2015analysis} (chap. 3) for a detailed description about various models for $T$.}. A second-order LPM can be derived by only taking into account the two dominant poles of the previous model.

An example of on-field measurements, depicted in Figure~\ref{fig:stickslip}, shows the effect of this torque $T$ on the angular speed. The periodic scheme which arises is called \textit{stick-slip}. It emerges because of the difference between the static and Coulomb friction coefficients making an antidamping on the torque function $T$. Even though the surface angular velocity seems not to vary much, there is a cycle for the downhole one and the angular speed is periodically close to zero, meaning that the bit is stuck to the rock. 

The stick-slip effect appears mostly when dealing with a low desired angular velocity $\Omega_0$ on a controlled drilling mechanism. Indeed, if the angular speed $\Phi_t(L, t)$ is small, the torque provided by the rotary table at $x = 0$ increases the torsion along the pipe. This increase leads to a higher $\Phi_t(L,t)$ but the negative damping on the torque function implies a smaller $T$. Consequently $\Phi_t(L,t)$ increases, this phenomenon is called the \textit{slipping} phase. Then, the control law reduces the torque in order to match $\Phi_t(L, t)$ to $\Omega_0$. Since the torque increases as well, that leads to a \textit{sticking} phase where $\Phi(L,t)$ remains close to $0$. A stick-slip cycle then emerges. Notice that this is not the case for high values of $\Omega_0$ since torque $T$ does not vary much with respect to $\Phi_t(L, t)$ making the system easier to control. In Figure~\ref{fig:stickslip}, one can see that the frequency of the oscillations is $0.17$Hz and its amplitude is between $10$ and $25$ radians per seconds.
\begin{figure}
	\centering
	\includegraphics[width=8.5cm]{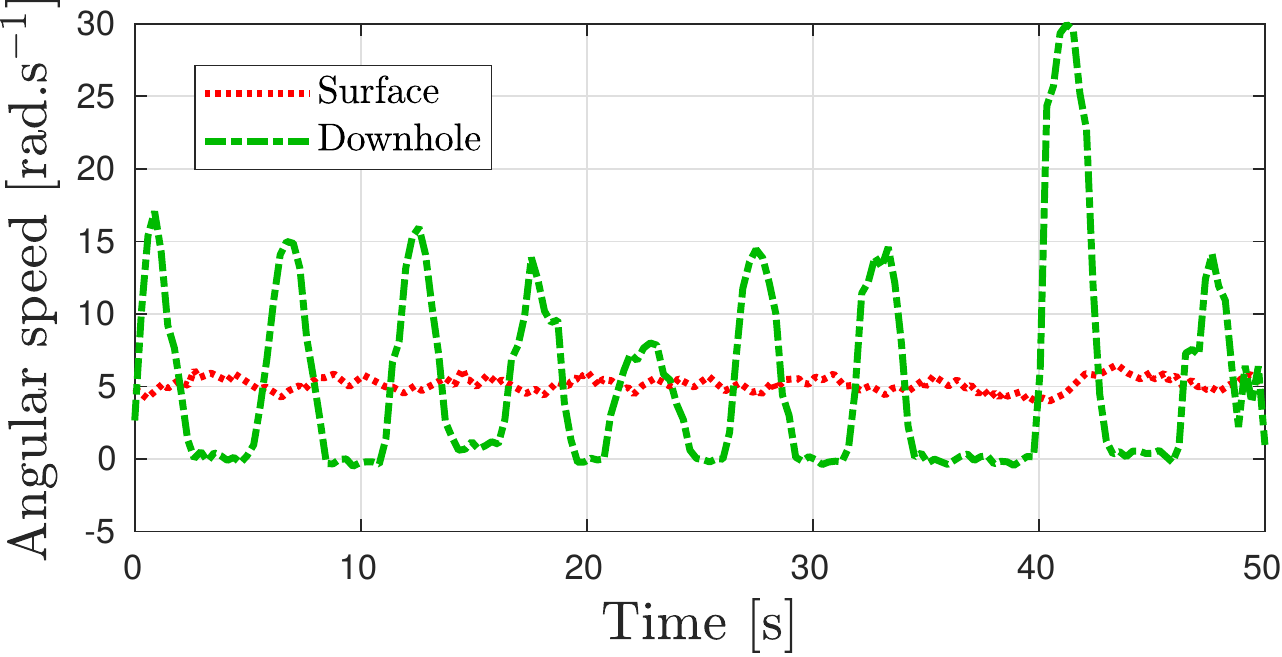}
	\caption{Nonlinear effect on the drilling mechanism due to the friction torque at the bottom of the pipe. These measurements are done on the field \cite{hinfty}. }
	\label{fig:stickslip}
\end{figure} 

Modeling this phenomenon is of great importance as friction effects are quite common when studying mechanical machinery.
In \cite{saldivar:hal-01425845}, the authors compare some models for $T$ and conclude that they produce very similar results. The main characteristic is a decrease of $T$ as $\Phi_t(L,t)$ increases. One standard model refers to the preliminary work of Karnopp \cite{karnopp1985computer} and Armstrong-Helouvry \cite{armstrong1990stick,armstrong1994survey} with an exponential decaying friction term as described in \cite{torqueModel2} for instance. This law is written thereafter where $\theta = \Phi_t(L, \cdot)$ is expressed in rad.s${}^{-1}$:
\begin{equation} \label{eq:torque}
	\begin{array}{l}
		T(\theta) = T_l(\theta) +  T_{nl}(\theta) , \\
		T_l(\theta) = c_b \theta, \\
		T_{nl}(\theta) = T_{sb} \left( \mu_{cb} + \left( \mu_{sb} - \mu_{cb} \right) e^{-\gamma_b |\theta|} \right) \sign(\theta).
	\end{array}
\end{equation}

This model has been used in \cite{1387580,marquez2015analysis} for instance. 

Notice that an on-field description of this mechanism applied in the particular context of drilling systems is provided in \cite{AARSNES2018712} and concludes that these models are fair approximations of the nonlinear phenomena visible in similar structures. 


\begin{figure}
	\centering
	\includegraphics[width=8.5cm]{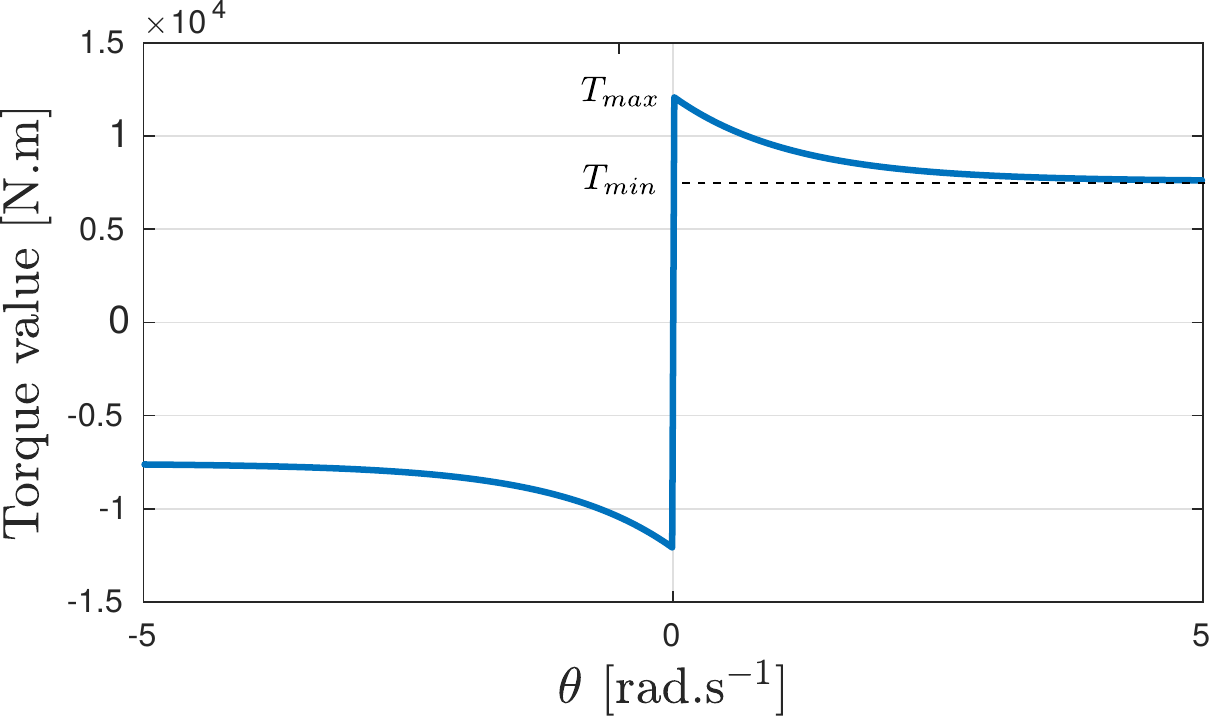}
	\caption{Nonlinear part of the torque. $T_{nl}$ is an approximation of the real friction torque coming from Karnopp's work \cite{karnopp1985computer}.}
	\label{fig:torque}
\end{figure} 

\begin{table}
	\centering
	\begin{tabular}{ll|c}
		\hline
		\multicolumn{2}{l|}{Parameter meaning} & Value \\
		\hline
		$I_r$ & Rotary table and drive inertia & $2122$ kg.m${}^2$ \\
		$I_b$ & Bit and drillstring inertia & $374$ kg.m${}^2$ \\
		$k$ & Drillstring stiffness & $1111$ N.m.rad${}^{-1}$ \\
		$\lambda_r$ & Coupled damping at top & $425$ N.m.s.rad${}^{-1}$\\
		$\lambda_b$ & Coupled damping at bottom & $23.2$ N.m.s.rad${}^{-1}$\\
		$d_r$ & Rotary table damping & $425$ N.m.s.rad${}^{-1}$ \\
		$d_b$ & Bit damping & $50$ N.m.s.rad${}^{-1}$ \\
		\hline
		$\gamma_b$ & Velocity decrease rate & $0.9$ s.rad${}^{-1}$ \\
		$\mu_{cb}$ & Coulomb friction coefficient & $0.5$ \\
		$\mu_{sb}$ & Static friction coefficient & $0.8$ \\
		$c_b$ & Bottom damping constant & $0.03$ N.m.s.rad${}^{-1}$ \\
		$T_{sb}$ & Static / Friction torque & $15~145$ N.m \\
	\end{tabular}
	\vspace{0.3cm}
	\caption{Parameters values and their meanings for the lumped parameter model taken from \cite{canudasdewit:hal-00394990,1387580,hinfty}.}
	\label{tab:lumpedMassModel}
\end{table}

At this stage, a lumped parameter model is interesting for its simplicity but does not take into account the infinite-dimensional nature of the problem and, as a consequence, is a good approximation in the case of small vibrations only \cite{saldivar2016control}. 
A deeper modeling can be done considering continuum mechanics and leading to a distributed parameter system. 

\subsection{Distributed parameter models (DPM)}

To tackle the finite-dimensional approximation of the previous model, another one derived from mechanical equations leads to a set of PDEs as described in the works \cite{challamel2000rock,weaver1990vibration}. This model has been enriched in \cite{AARSNES2016239, AARSNES2018712} where the system is presented from a control viewpoint and compared to on-field measurements. In the first papers, the model focuses on the propagation of the torsion only along the pipe. The axial propagation was introduced in the model by \cite{AARSNES2016239,saldivar:hal-01425845}. The new model is made up of two one-dimensional wave equations representing each deformation for $x \in (0, L)$ and $t > 0$:
\begin{subequations} \label{eq:distributed}
	\begin{align}
		\Phi_{tt}(x, t) = c_t^2 \Phi_{xx}(x, t) - \gamma_t \Phi_t(x, t), \label{eq:torsion}  \\
		\Psi_{tt}(x, t) = c_a^2 \Psi_{xx}(x, t) - \gamma_a \Psi_t(x, t), \label{eq:axial}
	\end{align}
	\label{eq:waves}
\end{subequations}
where again $\Phi$ is the twist angle, $\Psi$ is the axial movement, $c_t = \sqrt{G/\rho}$ is the propagation speed of the angle, $\gamma_t$ is the internal damping, $c_a = \sqrt{E/\rho}$ is the axial velocity and $\gamma_a$ is the axial distributed damping. A list of physical parameters and their values is given in Table~\ref{tab:parameters} and Figure~\ref{fig:drilling} helps giving a better understanding of the physical system. In other words, if $\Psi(\cdot, t) = 0$, then there is no compression in the pipe, meaning that the bit is not bouncing; if $\Phi_{tt}(\cdot, t) = 0$, then the angular speed along the pipe is the same, meaning that there is no increase or decrease of the torsion.

\begin{table}
	\centering
	\begin{tabular}{ll|c}
		\hline
		\multicolumn{2}{l|}{Parameter meaning} & Value \\
		\hline
		$L$ & Pipe length & $2000$m \\
		$G$ & Shear modulus & $79.3 \times 10^9$ N.m${}^{-2}$ \\
		$E$ & Young modulus & $200 \times 10^9$ N.m${}^{-2}$ \\
		$\Gamma$ & Drillstring's cross-section & $35 \times 10^{-4}$ m${}^4$ \\
		$J$ & Second moment of inertia & $1.19 \times 10^{-5}$ m${}^4$ \\
		$I_B$ & Bottom hole lumped inertia & $89$ kg.m${}^2$ \\
		$M_B$ & Bottom hole mass & $40~000$ kg \\
		$\rho$ & Density &  $8000$ kg.m${}^{-3}$ \\
		$g$ & Angular momentum & $2000$ N.m.s.rad${}^{-1}$ \\
		$h$ & Viscous friction coefficient & $200$ kg.s${}^{-1}$ \\
		$\gamma_a$ & Distributed axial damping & $0.69$ s${}^{-1}$ \\
		$\gamma_t$ & Distributed angle damping & $0.27$ s${}^{-1}$ \\
		$\delta$ & Weight on bit coefficient & $1$ m${}^{-1}$
	\end{tabular}
	\vspace{0.3cm}
	\caption{Physical parameters, meanings and their values \cite{AARSNES2016239,saldivar:hal-01425845}.}
	\label{tab:parameters}
\end{table}

For the previous model to be well-posed, top and bottom boundary conditions (at $x = 0$ and $x = L$) must be incorporated in \eqref{eq:distributed}. In this part, only the topside boundary condition is derived. There is a viscous damping at $x = 0$, and consequently a mismatch between the applied torque at the top and the angular speed. The topside boundary condition for the axial part is built on the same scheme and the following conditions are obtained for $t > 0$:
\begin{subequations} \label{eq:topCondition}
	\begin{align}
		G J \Phi_x(0, t) = g \Phi_t(0,t) - u_1(t), \label{eq:topTorsion} \\
		E \Gamma \Psi_x(0, t) = h \Psi_t(0, t) - u_2(t). \label{eq:topAxial}
	\end{align}
\end{subequations}

The downside boundary condition ($x = L$) is more difficult to grasp and is consequently derived later when dealing with a more complex model.


\subsection{Neutral-type time-delay model}

Studying an infinite-dimensional problem stated in terms of PDEs represents a relevant challenge. The equations obtained previously are damped wave equations, but, for the special case where $\gamma_a = \gamma_t = 0$, the system can be converted into a neutral time-delay system as done in \cite{marquez2015analysis}. This new formulation enables to use other tools to analyze its stability as the Lyapunov-Krasovskii Theorem or a frequency domain approach making its stability analysis slightly easier.

Nevertheless, the main drawback of this formulation is the assumption that the damping occurs at the boundary and not all along the pipe. This useful simplification, even if it encountered in many articles \cite{bresch2014output,marquez2015analysis,saldivar2016control}, is known to change in a significant manner the behavior of the system \cite{AARSNES2016239}. Indeed, it appears that without internal damping, the wave equation rephrases easily as a system of transport equations. It is then directly possible to observe with a delay of $c^{-1}$ at the top of the pipe what happened at the bottom of the pipe. This makes the control easier.

\subsection{Coupled ODE/PDE model}

To overcome the issue mentioned previously, a simpler model than the one derived in \eqref{eq:waves} is proposed in \cite{saldivar2016control}, where an harmonic oscillator is used to describe axial vibrations and the model results in a coupled ODE/PDE. \\
A second possibility, reported in \cite{challamel2000rock,marquez2015analysis} for example is to propose a second order ODE as the bottom boundary condition ($x = L$) for $t > 0$:
\begin{subequations} \label{eq:bottomCondition}
	\begin{align}
		GJ \Phi_{x}(L, t) = - I_B \Phi_{tt}(L, t) - T\left( \Phi_t(L, t) \right), \label{eq:bottomTorsion} \\
		E \Gamma \Psi_x(L,t) = - M_B \Psi_{tt}(L, t) - \delta T\left( \Phi_t(L, t) \right), \label{eq:bottomAxial}
	\end{align}
	\label{eq:bottom}
\end{subequations}
where $T$ represents the torque applied on the drilling bit by the rocks, described in equation~\eqref{eq:torque}. Notice that equation \eqref{eq:bottomTorsion} is coming from the conservation of angular momentum where $GJ \Phi_{x}(L, t)$ is the torque coming from the top of the pipe. Equation \eqref{eq:bottomAxial} is the direct application of Newton's second law of motion where $E \Gamma \Psi_x(L, t)$ is the force transmitted from the top to the bit and $\delta T\left( \Phi_t(L, t) \right)$ is the weight on bit due to the rock interaction.
Since \eqref{eq:bottom} is a second order in time differential equation, note that \eqref{eq:waves} together with \eqref{eq:bottom} indeed leads to a coupled ODE/PDE. 

There exist other bottom boundary conditions leading to a more complex coupling between axial and torsional dynamics. They nevertheless introduce delays which requires to have a better knowledge of the drilling bit. To keep the content general, the boundary conditions \eqref{eq:bottom} used throughout this paper is proposed accordingly with \cite{challamel2000rock,saldivar2016control,ATJECC}. \\
As a final remark, using some transformations based on \eqref{eq:torsion}, \eqref{eq:topTorsion} and \eqref{eq:bottomTorsion}, it is possible to derive a system for which backstepping controllers can be used \cite{bresch2014output,sagert2013backstepping}. This is the main reason why this model is widely used today. 

\subsection{Models comparison}

We propose in this subsection to compare the coupled ODE/PDE model and the lumped parameter models for the torsion only. We consider here a linearization of the system for large $\Omega_0$ and consequently we neglect the stick-slip effect by setting $T = 0$. 

First, denote by $\mathcal{H}_{DPM}$ the transfer function from $u_1$ to $\Phi(L,\cdot)$ for the DPM and $\mathcal{H}_{LPM}$ from $u_1$ to $\Phi_b$ for the LPM. We also define by $\mathcal{H}_{LPM2}$ a truncation of $\mathcal{H}_{LPM}$ considering only the two dominant poles. The Bode diagrams of $\mathcal H_{DPM}$, $\mathcal H_{LPM}$ and $\mathcal{H}_{LPM2}$ are drawn in Figure~\ref{fig:bodeModels}. 

Clearly, the LPMs catch the behavior of the DPM at steady states and low frequencies until the resonance, occurring around $\sqrt{k / I_b}$ rad.s${}^{-1}$. From a control viewpoint, the DPM has infinitely many harmonics as it can be seen on the plots but of lower magnitudes and damped as the frequency increases (around $-10$ dB at each decade). The magnitude plots is not sufficient to make a huge difference between the three models. Nevertheless, considering the phase, we see a clear difference. It appears that the DPM crosses the frequency $-180$ many times making the control margins quite difficult to assess. Moreover, that shows that the DPM is harder to control because of the huge difference of behavior after the resonance. They may consequently has a very different behavior when controlled. That is why we focus in this study on the the DPM, even if it is far more challenging to control than the LPMs.

\begin{figure}
	\centering
	\includegraphics[width=8.5cm]{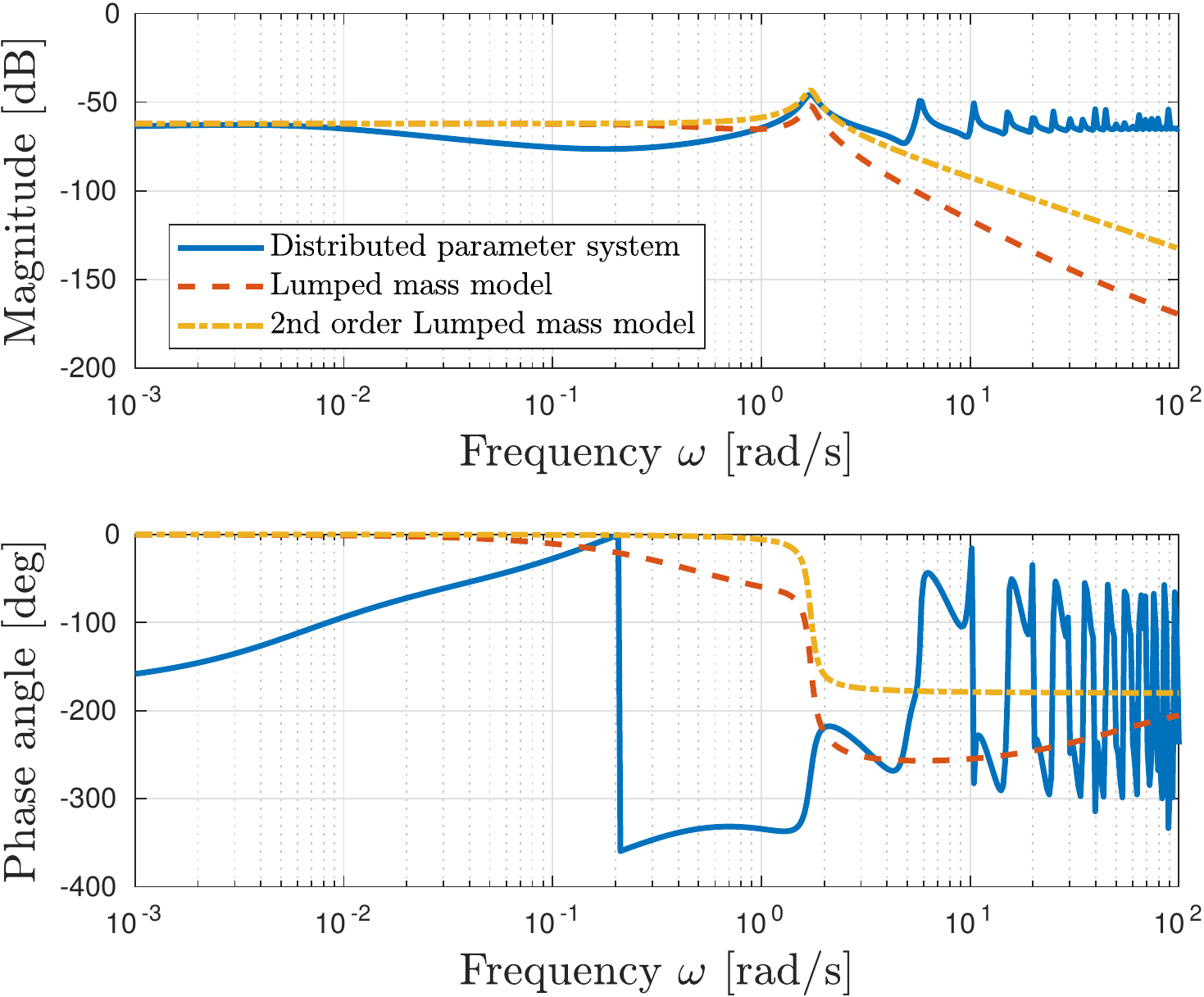}
	\caption[Bode diagram of the linearized models]{Bode diagram of $\mathcal{H}_{DPM}$, $\mathcal{H}_{LPM}$ and $\mathcal{H}_{LPM2}$.}
	\label{fig:bodeModels}
\end{figure} 

\section{Problem statement}

The coupled nonlinear ODE/PDE model derived in the previous subsection can be written as follows for $t > 0$ and $x \in [0, 1]$:
\begin{equation} \label{eq:problem}
	\left\{
		\begin{array}{l}
			\phi_{tt}(x,t) = \tilde{c}_t^2 \phi_{xx}(x,t) - \gamma_t \phi_t(x,t), \\
			\phi_x(0,t) = \tilde{g} \phi_t(0,t) - \tilde{u}_1(t), \\
			\phi_t(1,t) = z_1(t), \\
			\dot{z}_1(t) = - \alpha_1 \phi_x(1,t) - \alpha_2 T(z_1(t)),
		\end{array}
	\right.
\end{equation}
\begin{equation} \label{eq:problem2}
	\left\{
		\begin{array}{l}
			\psi_{tt}(x,t) = \tilde{c}_a^2 \psi_{xx}(x,t) - \gamma_a \psi_t(x,t), \\
			\psi_x(0,t) = \tilde{h} \psi_t(0,t) - \tilde{u}_2(t), \\
			\psi_t(1,t) = y_1(t), \\
			\dot{y}_1(t) = -\beta_1 \psi_x(1,t) - \beta_2 T(z_1(t)),
		\end{array}
	\right.
\end{equation}
where the normalized parameters are given in Table~\ref{tab:normParameters}. Note that the range for the spacial variable is now $x \in [0, 1]$ to ease the calculations. The initial conditions are as follows:
\[
	\left\{
	\begin{array}{ll}
		\phi(x,0) = \phi^0(x), & \phi_t(x,0) = \phi^1(x), \\
		\psi(x, 0) = \psi^0(x), & \psi_t(x,0) = \psi^1(x), \\
		z_1(0) = \phi^1(1), & y_1(0) = \psi^1(1).
	\end{array} 
	\right.
\]


\begin{table}
	\centering
	\begin{tabular}{ccc}
		\hline
		Parameter & Expression & Value \\
		\hline
		$\phi(x,t)$ & $\Phi(x L,t)$ & - \\
		$\psi(x,t)$ & $\Psi(x L,t)$ & - \\
		$\tilde{c}_t$ & $c_t L^{-1}$ & $1.57$ \\
		$\tilde{c}_a$ & $c_a L^{-1}$ & $2.5$ \\
		$\alpha_1$ & $\cfrac{GJ}{L I_B}$ & $5.3$ \\
		$\beta_1$ & $ \cfrac{E \Gamma}{L M_B}$ & $8.75$ \\
		$\alpha_2$ & $I_B^{-1}$ & $1.12 \cdot 10^{-2}$ \\
		$\beta_2$ & $\cfrac{\delta}{M_B}$ & $2.5 \cdot 10^{-5}$ \\
		$\tilde{g}$ & $\cfrac{g}{GJ}$ & $2.1 \cdot 10^{-3}$ \\
		$\tilde{h}$ & $\cfrac{h}{E \Gamma}$ & $2.86 \cdot 10^{-7}$ \\
		$\tilde{u}_1(t)$ & $\cfrac{1}{GJ} u_1(t)$ & - \\
		$\tilde{u}_2(t)$ & $\cfrac{1}{E \Gamma} u_2(t)$ & - \\
	\end{tabular}
	\vspace{0.3cm}
	\caption{Normalized parameters.}
	\label{tab:normParameters}
\end{table}

Since the existence and uniqueness of a solution to the previous problem is not the main contribution of this paper, it is assumed in the sequel. This problem has been widely studied (see \cite{basturk2017observer,bresch2014output,sagert2013backstepping,saldivar2016control,ATJdrilling} among many others) and the solution belongs to the following space if the initial conditions $(\phi^0, \phi^1, \psi^0, \psi^1)$ satisfies the boundary conditions (see \cite{bastin2016stability} for more details):
\[
	\begin{array}{l}
		\mathbb{X} = H^1 \times L^2, \ \mathbb{X}_1 = H^2 \times H^1, \\
		(\phi, \phi_t, \psi, \psi_t, z_1, y_1) \in C^0 (\mathbb{X}_1^2 \times \mathbb{R}^2).
	\end{array}
\]

\begin{remark} One may note that $\theta \mapsto T_{nl}(\theta)$ is not well defined for $\theta = 0$ because of the sign function. Nevertheless, since the nonlinearity acts directly on the variable $z$, it follows that there exists a unique solution to the ODE system in the sense of Filipov \cite{filippov1967classical}. A more detailed discussion on this point is provided in \cite{bisoffi2017global}.\end{remark}

System~\eqref{eq:problem}-\eqref{eq:problem2} is a cascade of two subsystems:
\begin{enumerate}
	\item System \eqref{eq:problem} is a coupled nonlinear ODE/string equation describing the torsion angle $\phi$. 
	\item System \eqref{eq:problem2} is a coupled linear ODE/string equation subject to the external perturbation $T(z_1)$ for $\Psi$.
\end{enumerate}
It appears clearly that the perturbation on the second subsystem depends on the first subsystem in $\phi$. Since they are very similar, the same analysis than the one conducted in this paper applies for the second subproblem. That is why we only study the evolution of the torsion.

\section{Exponential Stability of the Linear System}

System \eqref{eq:problem} is a nonlinear system because of the friction term $T_{nl}$ introduced in \eqref{eq:torque}. Nevertheless, for a high desired angular speed $\Omega_0$, $T_{nl}$ can be assumed constant as seen in Figure~\ref{fig:torque}. 
Moreover, studying this linear system can be seen as a first step before studying the nonlinear system, which relies mostly on the stability theorem derived in this section.

The proposed linear model of $T$ around $\Omega_0 \gg 1$ is 
\begin{equation} \label{eq:torqueLin}
	T(\theta) = c_b \theta + T_{nl}(\Omega_0) = c_b \theta + T_0.
\end{equation}
For high values of $\Omega_0$, $T_0 = T_{nl}(\Omega_0)$ is close to $T_{smooth}(\Omega_0)$ and at the limit when $\Omega_0$ tends to infinity, they are equal. Hence the nonlinear friction term for relatively large angular velocity does not influence much the system.

In our case, the proposed controller for this problem has been studied in a different setting in \cite{canudasdewit:hal-00394990,christoforou2003fully,ATJECC} and is a simple proportional/integral controller based on the single measurement of the angular velocity at the top of the drill (i.e. $\phi_t(0,t)$). The following variables are therefore introduced:
\begin{equation} \label{eq:controller}
	\begin{array}{l}
		\tilde{u}_1(t) = - k_p \left( \phi_t(0,t) - \Omega_0 \right) - k_i z_2(t), \\
		\dot{z}_2(t) = \phi_t(0,t) - \Omega_0,
	\end{array}
\end{equation}
where $k_p$ and $k_i$ are the gains of the PI controller. Combining equations \eqref{eq:problem}, \eqref{eq:torqueLin} and \eqref{eq:controller} leads to:

\begin{equation} \label{eq:problemLin}
	\left\{
		\begin{array}{l}
			\phi_{tt}(x,t) = c^2 \phi_{xx}(x,t) - \gamma_t \phi_t(x,t), \\
			\phi_x(0,t) = (\tilde{g} + k_p) \phi_t(0,t) - k_p \Omega_0 + k_i C_2 Z(t), \\
			\phi_t(1,t) = C_1Z (t), \\
			\dot{Z}(t) = A Z(t) + B \left[ \begin{smallmatrix} \phi_t(0,t) \\ \phi_x(1,t) \end{smallmatrix} \right] + B_2 \left[ \begin{smallmatrix} \Omega_0 \\ T_0 \end{smallmatrix} \right],
		\end{array}
	\right.
\end{equation}
where 
\[
	\begin{array}{lll}
		A \!=\! \left[ \begin{matrix} - \frac{c_b}{I_B} & 0 \\ 0 & 0 \end{matrix} \right], & B_{\hphantom{1}} \!=\! \left[ \begin{matrix} 0 & -\alpha_1 \\ 1 & 0 \end{matrix} \right], & B_2 \!=\! \left[ \begin{matrix} 0 & -\alpha_2 \\ - 1 & 0 \end{matrix} \right], \\
		Z \!=\! \left[ \begin{matrix} z_1 & z_2 \end{matrix} \right]^{\top}\!\!\!\!, & C_1 = \left[ \begin{matrix} 1 & 0 \end{matrix} \right], & C_2 = \left[ \begin{matrix} 0 & 1 \end{matrix} \right].
	\end{array}
\]
To ease the notations, $c = \tilde{c}_t$. We denote by 
\[
	X = (\phi_x, \phi_t, z_1, z_2) \in C^1([0, \infty), \mathcal{H})
\]
with $\mathcal{H} = L^2 \times L^2 \times \mathbb{R}^2$ the infinite dimensional state of system~\eqref{eq:problemLin}. The control objective in the linear case is to achieve the exponential stabilization of an equilibrium point of \eqref{eq:problemLin} in angular speed, i.e. $\phi_t(1)$ is going exponentially to a given constant reference value $\Omega_0$. To this extent, the following norm on $\mathcal{H}$ is introduced:
\[
	\| X \|_{\mathcal{H}}^2 = z_1^2 + z_2^2 + c^2 \| \phi_x(\cdot) \|^2 + \| \phi_t(\cdot) \|^2.
\]
The definition of an equilibrium point of \eqref{eq:problemLin} and its exponential stability follows from the previous definitions.

\begin{definition}$X_{\infty} \in \mathcal{H}$ is an \textbf{equilibrium point} of \eqref{eq:problemLin} if for the trajectory $X \in C^1([0, \infty), \mathcal{H})$ of \eqref{eq:problemLin} with initial condition $X_{\infty}$, the following holds:
\[
	\forall t > 0, \quad \| \dot{X}(t) \|_{\mathcal{H}} = 0.
\]
\end{definition}

\begin{definition}\label{def:stabExpo} Let $X_{\infty}$ be an equilibrium point of \eqref{eq:problemLin}. $X_{\infty}$ is said to be \textbf{$\mu$-exponentially stable} if
\[
	\forall t > 0, \quad \| X(t) - X_{\infty} \|_{\mathcal{H}} \leq \gamma \ \| X_0 - X_{\infty} \|_{\mathcal{H}} \ e^{-\mu t}
\]
holds for $\gamma \supeq 1$, $\mu > 0$ and for any initial conditions $X_0 \in \mathcal{H}$ satisfying the boundary conditions. Here, $X$ is the trajectory of \eqref{eq:problemLin} whose initial condition is $X_0$. 

Hence, $X_{\infty}$ is said to be \textbf{exponentially stable} if there exists $\mu > 0$ such that $X_{\infty}$ is $\mu$-exponentially stable.
\end{definition}

Before stating the main result of this part, a lemma about the equilibrium point is proposed.

\begin{lemma}\label{lem:equilibirum} Assume $k_i \neq 0$, then there exists a unique equilibrium point $X_{\infty} = (\phi^{\infty}_x, \phi^{\infty}_t, z_1^{\infty}, z_2^{\infty}) \in \mathcal{H}$ of \eqref{eq:problemLin} and it satisfies $\phi^{\infty}_t = \Omega_0$.
\end{lemma}
\begin{proof}An equilibrium point $X_{\infty}$ of \eqref{eq:problemLin} is such that: $(\phi^{\infty}_{xt}, \phi^{\infty}_{tt}, \frac{d}{dt} z^{\infty}_1, \frac{d}{dt} z^{\infty}_2) = 0$. \\
Since $\frac{d}{dt} z^{\infty}_2 = \phi_t^{\infty}(0, t) - \Omega_0= 0$ and $\partial_x \phi_{t}^{\infty} = 0$ and $\partial_t \phi_t^{\infty} = 0$, $\phi^{\infty}_t = \Omega_0$ holds.\\
We also get from $\phi_{xx}^{\infty} = 0$ and $\partial_t \phi_x^{\infty} = 0$ that $\phi^{\infty}_x$ is a first order polynomial in $x$. Together with the boundary conditions, this system has a unique solution if $k_i \neq 0$ such as:
\[
	X_{\infty} = \left( \phi_x^{\infty}, \Omega_0, \Omega_0, \frac{\phi^{\infty}_x(0) - \tilde{g} \Omega_0}{k_i} \right),
\]
where $\phi_x^{\infty}(x) = \frac{\gamma_t \Omega_0}{c^2}(x-1) - \frac{c_b}{\alpha_1 I_B} \Omega_0 - \frac{\alpha_2}{\alpha_1} T_0$ for $x \in [0, 1]$.
\end{proof}


\subsection{Exponential stability of the closed-loop system~\eqref{eq:problemLin}}

The main result of this part is then stated as follows.
\begin{theo} \label{theo:StabLin}
Let $N \in \mathbb{N}$. Assume there exists $P_N \in \mathbb{S}^{2+2(N+1)}$, $R = \diag(R_1, R_2) \succeq 0, S = \diag(S_1, S_2) \succ 0$, $Q \in \mathbb{S}^2_+$ such that the following LMIs hold:
\begin{equation} \label{eq:LMI}
	\begin{array}{l}
		\Theta_N = c \hspace{0.1cm} \Theta_{1,N} + \Theta_{2,N} - Q_N \prec 0, \\
		P_N + S_N \succ 0,
	\end{array}
\end{equation}
and
\begin{equation} \label{eq:extraCondition}
	\begin{array}{l}
		\Gamma_0 = c R + \frac{\gamma_t}{2} U_0 - Q \succeq 0, \\
		\Gamma_1 = c R + \frac{\gamma_t}{2} U_1 - Q \succeq 0,
	\end{array}
\end{equation}
where
\begin{equation*} 
	\begin{array}{cl}
		\!\!\!\!\!\! \Theta_{1,N} \!\!\!\!& = H_N^{\top} \left[\begin{smallmatrix} S_1 + R_1 & 0 \\ 0 & - S_2 \end{smallmatrix}\right] H_N - G_N^{\top} \left[ \begin{smallmatrix} S_1 & 0 \\ 0 & -S_2-R_2 \end{smallmatrix} \right] G_N, \\
		\!\!\!\!\!\! \Theta_{2,N} \!\!\!\!& = \He \left( D_N^{\top} P_N F_N \right), \\
	\end{array}
\end{equation*}
\begin{equation} \label{eq:defTheo}
	\begin{array}{cl}
		\!\!\!\!\!\! F_N \!\!\!\! &= \left[ \begin{matrix} I_{2+2(N+1)} & 0_{2+2(N+1), 2} \end{matrix} \right], \\
		\!\!\!\!\!\! D_N \!\!\!\! &= \left[ \begin{matrix} J_N^{\top} & c M_{N}^{\top} \end{matrix} \right]^{\!\top}\!\!\!,  
		\quad J_N = \left[ \begin{matrix} A & 0_{2,2(N+1)} & B \end{matrix} \right], \\
		\!\!\!\!\!\! M_{N} \!\!\!\! &= \mathbb{1}_N H_{N}\! -\! \bar{\mathbb{1}}_N G_{N}\! -\! \left[\begin{matrix}  0_{2(N + 1),2} &\!\!\!\! L_{N} &\!\!\!\! 0_{2(N + 1), 2}\end{matrix} \right], \\
		\!\!\!\!\!\! U_0 \!\!\!\! &= \left[ \begin{smallmatrix} 2 S_1 & S_1 + S_2 + R_2 \\ S_1 + S_2 + R_2 & 2 (S_2 + R_2) \end{smallmatrix} \right], \ U_1 = \left[ \begin{smallmatrix} 2 (S_1 + R_1) & S_1 + S_2 + R_1 \\ S_1 + S_2 + R_1 & 2 S_2 \end{smallmatrix} \right], \\
	\end{array}
\end{equation}
\begin{equation*}
	\begin{array}{cl}
		\!\!\!\!\!\! G_{N} \!\!\!\! &= \left[ \begin{matrix} \begin{smallmatrix} c k_i C_2 \\  -c k_i C_2 \end{smallmatrix} & 0_{2, 2(N+1)} & G \end{matrix} \right], \quad G = \left[ \begin{smallmatrix} 1 + c(\tilde{g} + k_p) & 0 \\ 1 - c(\tilde{g} + k_p) & 0 \end{smallmatrix} \right], \\
		\!\!\!\!\!\! H_{N} \!\!\!\! &= \left[ \begin{matrix} \begin{smallmatrix} C_1 \\  C_1 \end{smallmatrix} & 0_{2, 2(N+1)} & H \end{matrix} \right],  \quad \ \ H = \left[ \begin{smallmatrix} 0 & c  \\ 0 &  -c \end{smallmatrix} \right], \\
		\!\!\!\!\!\! Q_N \!\!\!\!  &= \diag(0_2, Q, 3 Q, \cdots, (2N+1) Q, 0_{2}), \\
		\!\!\!\!\!\! S_N \!\!\!\!   &= \diag \left( 0_2, S, 3S, \cdots, (2N+1)S \right), \\
		\!\!\!\!\!\! L_N  \!\!\!\!  &= \displaystyle \left[ \ell_{j,k} \Lambda \right]_{j,k \in [0, N]}
		 - \frac{\gamma_t}{2} \diag \left( \left[ \begin{smallmatrix} 1 & 1 \\ 1 & 1 \end{smallmatrix} \right], \dots, \left[ \begin{smallmatrix} 1 & 1 \\ 1 & 1 \end{smallmatrix} \right] \right), \\
		\!\!\!\!\!\! \mathbb{1}_N  \!\!\!\!  &=\left[\begin{smallmatrix} \Lambda  \\ \vdots \\ \Lambda\end{smallmatrix}\right], \quad
		\bar{\mathbb 1}_N = \left[\begin{smallmatrix} \Lambda \\  \vdots \\(-1)^{N} \Lambda \end{smallmatrix}\right], \quad \Lambda = \left[ \begin{matrix} c & 0 \\ 0 & -c \end{matrix} \right],
	\end{array}
\end{equation*}
and 
\[
	\ell_{k,j} = \left\{
		\begin{array}{ll}
			(2j+1)(1 - (-1)^{j+k}), & \text{ if } j \infeq k, \\
			0 & \text{ otherwise},
		\end{array}
	\right.
\]
then the equilibrium point of system \eqref{eq:problemLin} is exponentially stable. As a consequence, $\lim_{t \to +\infty} | \phi_t(1,t) - \Omega_0 | = 0$.
\end{theo}

The proof is given thereafter but some practical consequences and preliminary results are firstly derived.

\begin{remark} \label{rem:hierarchy} The methodology used to derive the previous result has been firstly introduced in \cite{seuret:hal-01065142} to deal with time-delay systems. It has been proved in the former article that this theorem provides a hierarchy of LMI conditions. That means if the conditions of Theorem~\ref{theo:StabLin} are met for $N =N_1 \supeq 0$, then the conditions are also satisfied for all $N > N_1$. Then, the LMIs provide a sharper analysis as the order $N$ of the theorem increases. Nevertheless, the price to pay for a more precise analysis is the increase of the number of decision variables and consequently a higher computational burden. It has been noted in \cite{seuret:hal-01065142} for a time-delay system and in \cite{besselString} for a coupled ODE/string equation system that very sharp results are obtained even for low orders $N$. \end{remark}

As we aim at showing that $X_{\infty}$ is an exponentially stable equilibrium point of system \eqref{eq:problemLin} in the sense of Definition~\ref{def:stabExpo}, the following variable is consequently defined for $t \geq 0$:
\[
	\tilde{X}(t) = X(t) - X_{\infty} = \left( \tilde{\phi}_x(t), \tilde{\phi}_t(t), \tilde{z}_1(t), \tilde{z}_2(t) \right),
\]
where $X$ is a trajectory of system~\eqref{eq:problemLin}. \\
The following lemma, given in \cite{barreauCDC}, provides a way for estimating the exponential decay-rate of system~\eqref{eq:problemLin} as soon as a Lyapunov functional $V$ is known.

\begin{lemma} \label{lemma:epsilon} Let $V$ be a Lyapunov functional for system \eqref{eq:problemLin} and $\mu \geq 0$. Assume there exist $\varepsilon_1, \varepsilon_2, \varepsilon_3 > 0$ such that the following holds:
\begin{equation} \label{eq:epsilon}
	\left\{
	\begin{array}{l}
		 \varepsilon_1 \| \tilde{X} \|^2_{\mathcal H} \infeq V(\tilde{X}) \infeq \varepsilon_2 \| \tilde{X} \|^2_{\mathcal H}, \\
		\dot{V}(\tilde{X}) + 2 \mu V(\tilde{X}) \infeq - \varepsilon_3 \| \tilde{X} \|^2_{\mathcal H},
	\end{array}
	\right.
\end{equation}
then the equilibrium point of system \eqref{eq:problemLin} is $\mu$-exponentially stable. If $\mu = 0$, then it is exponentially stable.
\end{lemma}

\begin{proof}[Proof of Theorem~\ref{theo:StabLin}]
For simplicity, the following notations are used throughout the paper for $k \in \mathbb{N}$: 
\begin{equation} \label{eq:chi}
	\begin{array}{ll}
		\tilde{\chi}(x) = \left[ \begin{matrix} \tilde{\phi}_t(x) + c \tilde{\phi}_x(x) \\ \tilde{\phi}_t(x) - c \tilde{\phi}_x(x) \end{matrix} \right], &
		 \displaystyle \tilde{\Chi}_k(t) = \int_0^1 \tilde{\chi}(x,t) \mathcal{L}_k(x) dx,
	\end{array}
\end{equation}
where $\left\{ \mathcal{L}_k \right\}_{k \in \mathbb{N}}$ is the orthogonal family of Legendre polynomials as defined in Appendix~\ref{sec:bessel}. $\tilde{\Chi}_k$ is then the projection coefficient of $\tilde{\chi}$ along the Legendre polynomial of degree $k$. $\tilde{\chi}$ refers to the Riemann coordinates and presents useful properties as discussed in \cite{besselString} for instance.

\subsubsection{Choice of a Lyapunov functional candidate}

The proposed Lyapunov functional is inspired from \cite{barreauCDC,besselString,safi:hal-013540732} and is divided into two parts as follows:
\begin{equation} \label{eq:lyap}
	V_{N}(\tilde{X}) = \tilde{Z}_N^{\top} P_N \tilde{Z}_N + \V (\tilde{\chi})
\end{equation}
where $\tilde{Z}_N = \left[ \tilde{z}_1 \ \tilde{z}_2 \ \tilde{\Chi}_0^{\top} \ \cdots \ \tilde{\Chi}_{N}^{\top} \right]^{\top}\!\!$ and
\[
	\V (\tilde{\chi}) = \int_0^1 \tilde{\chi}^{\top}(x) \left[ \begin{smallmatrix} S_1 + x R_1 & 0 \\ 0 & S_2 + (1-x) R_2 \end{smallmatrix} \right] \tilde{\chi}(x) dx.
\] 
$\V$ is a traditional Lyapunov functional candidate for a transport system as shown in \cite{coron2007control}.

\subsubsection{Exponential stability} \ 

\textit{Existence of $\varepsilon_1$}: This part is inspired by \cite{barreauCDC}. The inequalities $P_N + S_N \succ 0, R \succeq 0$ imply the existence of $\varepsilon_1 > 0$ such that:
\begin{equation} \label{eq:epsilon1}
	\begin{array}{rcl}
		P_N + S_N & \succeq & \varepsilon_1 \diag \left( I_2, \frac{1}{2} \mathcal{I}_{N} \right), \\
		S & \succeq & \frac{\varepsilon_1}{2} I_2,
	\end{array}
\end{equation}
with $\mathcal{I}_N = \diag \left\{ (2k+1) I_2\right\}_{k \in [0, N]}$. This statement implies the following on $V_{N}$:
\[
	\begin{array}{rl}
		\!\!\!V_{N}(\tilde{X}) \!\!\!\!&\displaystyle\geq \tilde{Z}_N^{\top} (P_N + S_N) \tilde{Z}_N\! - \!\sum_{k=0}^N (2k+1) \tilde{\Chi}_k^{\top} S \tilde{\Chi}_k \quad \quad \quad \\
		& \displaystyle \hfill + \int_0^1 \!\! \tilde{\chi}(x)^{\top}\! \left( S-\frac{\varepsilon_1}{2} I_2 \right) \tilde{\chi}(x) dx + \frac{\varepsilon_1}{2} \| \tilde{\chi} \|^2 \\
		& \geq \varepsilon_1 \left( \tilde{z}_1^2 + \tilde{z}_2^2 + \frac{1}{2} \| \tilde{\chi} \|^2 \right) \\
		& \displaystyle \hfill - \sum_{k=0}^N (2k+1) \tilde{\Chi}_k^{\top} \tilde{S} \tilde{\Chi}_k + \int_0^1 \tilde{\chi}^{\top}(x) \tilde{S} \tilde{\chi}(x) dx \\
		& \geq \varepsilon_1 \| \tilde{X} \|_{\mathcal{H}}^2,
	\end{array}
\]
where $\tilde{S} = S - \frac{\varepsilon_1}{2} I_2$ which ends the proof of existence of $\varepsilon_1$.

\textit{Existence of $\varepsilon_2$}: Following the same line as previously, the following holds for a sufficiently large $\varepsilon_2 > 0$:
\[
	\begin{array}{rcl}
		P_N & \preceq & \varepsilon_2 \diag \left( I_2, \frac{1}{4} \mathcal{I}_N \right), \\
		\left[ \begin{smallmatrix} S_1 + xR_1 & 0 \\ 0 & S_2 + (1-x)R_2 \end{smallmatrix} \right] 
		& \preceq & \frac{\varepsilon_2}{4} I_2.
	\end{array}
\]
Using these inequalities in \eqref{eq:lyap} leads to:
\[
	\begin{array}{rl}
		V_{N}(\tilde{X}) \!\!\! & \displaystyle \leq \varepsilon_2 \left( \tilde{z}_1^2 + \tilde{z}_2^2 + \sum_{k=0}^N \frac{2k+1}{4} \tilde{\Chi}_k^{\top} \tilde{\Chi}_k + \frac{1}{4} \| \tilde{\chi} \|^2 \right) \\
		& \leq \varepsilon_2 \left(  \tilde{z}_1^2 + \tilde{z}_2^2 + \frac{1}{2} \| \tilde{\chi} \|^2 \right) = \varepsilon_2 \| \tilde{X} \|_{\mathcal{H}}^2.
	\end{array}
\]
The last inequality is a direct application of Lemma~\ref{lemma:bessel}.

\textit{Existence of $\varepsilon_3$}: The time derivation of $\tilde{\chi}$ leads to:
\[
	\tilde{\chi}_t(x) = \Lambda \tilde{\chi}_x(x) - \gamma_t \left[ \begin{smallmatrix} 1 \\ 1 \end{smallmatrix} \right] \tilde{\phi}_t(x).
\]
Noting that $\tilde{\phi}_t(x) = \frac{1}{2} \left[ \begin{smallmatrix} 1 & 1 \end{smallmatrix} \right] \tilde{\chi}(x)$, we get:
\begin{equation} \label{eq:chi_t}
	\tilde{\chi}_t(x) = \Lambda \tilde{\chi}_x(x) - \frac{\gamma_t}{2} \left[ \begin{smallmatrix} 1 & 1 \\ 1 & 1 \end{smallmatrix} \right] \tilde{\chi}(x).
\end{equation}
The derivative of $\V$ along the trajectories of \eqref{eq:problem} leads to:
\[
	\dot{\V}(\tilde{\chi}) = 2c \V_1(\tilde{\chi}) - \frac{\gamma_t}{2} \V_2(\tilde{\chi}),
\]	
with
\[
	\begin{array}{cl}
		\V_1(\tilde{\chi}) \!\!\!&= \displaystyle \int_0^1 \tilde{\chi}^{\top}_x(x) \left[\begin{smallmatrix} S_1 + x R_1 & 0 \\ 0 & -S_2 - (1-x) R_2 \end{smallmatrix}\right] \tilde{\chi}(x) dx,\\	
		\V_2(\tilde{\chi}) \!\!\!&= \displaystyle \int_0^1 \tilde{\chi}^{\top}(x) U(x) \tilde{\chi}(x) dx, \\
		U(x) \!\!\!&= \left[\begin{smallmatrix} 2 (S_1 + x R_1)  & S_1 + S_2 + x R_1 + (1-x) R_2 \\ S_1 + S_2 + x R_1 + (1-x) R_2 & 2 ( S_2 + (1-x) R_2) \end{smallmatrix}\right].
	\end{array}
\]

An integration by part on $\V_1$ shows that:
\begin{multline*}
	2\V_1(\tilde{\chi}) = \tilde{\chi}^{\top}(1) \left[ \begin{smallmatrix} S_1 + R_1 & 0 \\ 0 & -S_2 \end{smallmatrix} \right] \tilde{\chi}(1)  \\
	- \tilde{\chi}^{\top}(0) \left[ \begin{smallmatrix} S_1 & 0 \\ 0 & -S_2 - R_2 \end{smallmatrix} \right] \tilde{\chi}(0) \\
	- \int_0^1 \tilde{\chi}^{\top}(x) \left[ \begin{smallmatrix} R_1 & 0 \\ 0 & R_2 \end{smallmatrix} \right] \tilde{\chi}(x) dx.
\end{multline*}
The previous calculations lead to the following derivative:
\begin{multline}
	\dot{\V}(\tilde{\chi}) = c \left( \tilde{\chi}^{\top}(1) \left[ \begin{smallmatrix} S_1 + R_1 & 0 \\ 0 & -S_2  \end{smallmatrix} \right] \tilde{\chi}(1) \right.\\
	\left. - \tilde{\chi}^{\top}(0) \left[ \begin{smallmatrix} S_1 & 0 \\ 0 & -S_2 - R_2 \end{smallmatrix} \right] \tilde{\chi}(0) \vphantom{e^{\frac{2 \mu}{c}x}} \right)\\
	- \int_0^1 \tilde{\chi}^{\top}(x) \left( c R + \frac{\gamma_t}{2} U(x) \right) \tilde{\chi}(x) dx.
\end{multline}
\\
By a convexity argument, if \eqref{eq:extraCondition} is verified then $c R + U(x) \succeq Q \succ 0$ holds for $x \in [0, 1]$. Consequently, noticing that $\tilde{\chi}(0) = G_N \tilde{\xi}_N$, $\tilde{\chi}(1) = H_N \tilde{\xi}_N$, we get:
\begin{equation*}
	\dot{\V}(\tilde{\chi}) \leq c\hspace{0.05cm} \tilde{\xi}^{\top}_N \Theta_{1,N} \tilde{\xi}_N - \int_0^1 \tilde{\chi}^{\top}(x) Q \tilde{\chi}(x) dx,
\end{equation*}
with
\[
	\tilde{\xi}_N = \left[ \begin{matrix} \tilde{Z}_N^{\top} & \tilde{\phi}_t(0) & \tilde{\phi}_x(1) \end{matrix} \right]^{\top}.
\]

Using Lemma~\ref{lem:Chi_k} and the fact that $\dot{\tilde{Z}}_N = D_N \tilde{\xi}_N$ and $\tilde{Z}_N = F_N \tilde{\xi}_N$, we get the following:
\begin{equation} \label{eq:VNdotLinear}
	\begin{array}{rl}
 		\dot{V}_{N}(\tilde{X}) \!\!\!&= \He \left( \dot{\tilde{Z}}_N^{\top} P_N \tilde{Z}_N \right) + \dot{\V}(\tilde{\chi})  \\
 			& \displaystyle \leq \tilde{\xi}^{\top}_N \Theta_{N} \tilde{\xi}_N + \sum_{k=0}^N (2k+1) \tilde{\Chi}_k Q \tilde{\Chi}_k \quad \quad \quad \quad\\
 			& \displaystyle \hfill - \int_0^1 \tilde{\chi}^{\top}(x) Q \tilde{\chi}(x) dx,
	\end{array}
\end{equation}
with $\Theta_N$ defined in \eqref{eq:LMI}. Since $\Theta_N \!\prec\! 0$ and $Q \!\succ\! 0$, we get:
\[
	\begin{array}{rcl}
		\Theta_N & \preceq & - \varepsilon_3 \diag \left( I_2, \frac{1}{2} I_2, \frac{3}{2} I_2, \dots, \frac{2N+1}{2} I_2, 0_2 \right), \\
		Q & \succeq & \frac{\varepsilon_3}{2} I_2.
	\end{array}
\]
Then, $\dot{V}_{N}$ is upper bounded by:
\begin{equation*}
	\begin{array}{rl}
		\dot{V}_{N}(\tilde{X}) \!\!\!& \leq -\varepsilon_3 \left( \tilde{z}_1^2 + \tilde{z}_2^2 + \frac{1}{2} \| \tilde{\chi} \|^2 \right) \hspace{3cm} \\
		& \displaystyle \hfill + \sum_{k=0}^N (2k+1) \tilde{\Chi}_k^{\top} \left( Q - \frac{\varepsilon_3}{2} I_2 \right) \tilde{\Chi}_k \hspace{1cm} \\
		& \displaystyle \hfill - \int_0^1 \tilde{\chi}^{\top}(x) \left( Q - \frac{\varepsilon_3}{2} I_2 \right) \tilde{\chi}(x) dx \\
		& \leq - \varepsilon_3 \| \tilde{X} \|_{\mathcal{H}}^2.
	\end{array}
\end{equation*}
The last inequality comes from a direct application of Bessel's inequality \eqref{eq:bessel}.

\textit{Conclusion :} Using Lemma~\ref{lemma:epsilon}, we indeed get the exponential convergence of all trajectories of \eqref{eq:problemLin} to the desired equilibrium point.
\end{proof}

\subsection{Exponential stability with a guaranteed decay-rate}

It is possible to estimate the decay-rate $\mu$ of the exponential convergence with a slight modification of the LMIs as it is proposed in the following corollary.

\begin{coro}\label{coro:alpha} Let $N \in \mathbb{N}$, $\mu > 0$ and $\gamma_t \geq 0$. If there exist $P_N \in \mathbb{S}^{2+2(N+1)}$, $R = \diag(R_1, R_2) \succeq 0, S = \diag(S_1, S_2) \succ 0, Q \in \mathbb{S}^2_+$ such that \eqref{eq:extraCondition} and the following LMIs hold:
\begin{equation} \label{eq:thetaMu}
	\begin{array}{l}
		\Theta_{N, \mu} = \Theta_N + 2 \mu F_N^{\top} \left( P_N + S_N \right) F_N \prec 0, \\
		P_N + S_N \succ 0,
	\end{array}
\end{equation}
with the parameters defined as in Theorem~\ref{theo:StabLin}, then the equilibrium point of system~\eqref{eq:problemLin} is $\mu$-exponentially stable. \end{coro}

\begin{proof}
To prove the exponential stability with a decay-rate of at least $\mu > 0$, we use Lemma~\ref{lemma:epsilon}. Similarly to the previous proof, we have the existence of $\varepsilon_1, \varepsilon_2 > 0$. The existence of $\varepsilon_3$ is slightly different. First, note that for the Lyapunov functional candidate \eqref{eq:lyap}, we get:
\[
	\begin{array}{rl}
		\displaystyle V_N(\tilde{X}_{\phi}) \!\!\!\!&\displaystyle\geq \tilde{Z}_N^{\top} P_N \tilde{Z}_N + \int_0^1 \tilde{\chi}^{\top}(x) S \tilde{\chi}(x) dx \\
		&\displaystyle\geq \tilde{Z}_N^{\top} P_N \tilde{Z}_N + \sum_{k=0}^N (2k+1) \tilde{\Chi}_k^{\top} S \tilde{\Chi}_k.
	\end{array}
\]
This inequality was obtained using equation~\eqref{eq:bessel}. Using the notations of the previous theorem yields:
\begin{equation} \label{eq:VNapprox}
	V_N(\tilde{X}_{\phi}) \geq \tilde{\xi}_N^{\top} F_N^{\top} \left( P_N + S_N \right) F_N \tilde{\xi}_N.
\end{equation}
Coming back to \eqref{eq:VNdotLinear} and
using \eqref{eq:thetaMu} leads to:
\begin{multline*}
 	\dot{V}_{N}(\tilde{X}_{\phi}) \leq \tilde{\xi}^{\top}_N \Theta_{N,\mu} \tilde{\xi}_N - 2 \mu \tilde{\xi}_N^{\top} \left( F_N^{\top} P_N F_N + S_N \right) \tilde{\xi}_N \\ + \sum_{k=0}^N (2k+1) \tilde{\Chi}_k^{\top} Q \tilde{\Chi}_k  - \int_0^1 \tilde{\chi}^{\top}(x) Q \tilde{\chi}(x) dx.
\end{multline*}
Injecting inequality \eqref{eq:VNapprox} into the previous inequality leads to:
\begin{multline*}
 	\dot{V}_{N}(\tilde{X}_{\phi}) + 2 \mu V_N(\tilde{X}_{\phi}) \leq \tilde{\xi}^{\top}_N \Theta_{N,\mu} \tilde{\xi}_N  \\
 	+ \sum_{k=0}^N (2k+1) \tilde{\Chi}_k^{\top} Q \tilde{\Chi}_k  - \int_0^1 \tilde{\chi}^{\top}(x) Q \tilde{\chi}(x) dx.
\end{multline*}
Using the same techniques than for the previous proof leads to the existence of $\varepsilon_3 > 0$ such that:
\[
	\dot{V}_N(\tilde{X}_{\phi}) + 2 \mu V_N(\tilde{X}_{\phi}) \leq - \varepsilon_3 \| \tilde{X}_{\phi} \|_{\mathcal{H}}^2.
\]
Lemma~\ref{lemma:epsilon} concludes then on the $\mu$-exponential stability.
\end{proof}

\begin{remark} Wave equations can sometimes be modeled as neutral time-delay systems \cite{barreauInputOutput,marquez2015analysis}. This kind of system is known to possess some necessary stability conditions as noted in \cite{barreauInputOutput,bastin2016stability}. In the book \cite{bastin2016stability}, the following criterion is derived:
\begin{equation} \label{eq:alphamax}
	\mu \leq \frac{c}{2} \log \left| \frac{1+c (\tilde{g} + k_p)}{1-c (\tilde{g} + k_p)} \right| = \mu_{max}.
\end{equation}
This result implies that there exists a maximum decay-rate and if this maximum is negative, then the system is unstable. The LMI $\Theta_{N}^{\mu} \prec 0$ contains the same necessary condition, meaning that the neutral aspect of the system is well-captured. \end{remark}

\subsection{Strong stability against small delay in the control}
\label{sec:kp}

A practical consequence of the neutral aspect of system~\eqref{eq:problemLin} is that it is very sensitive to delay in the control \eqref{eq:controller}.
Indeed, if the control is slightly delayed, a new necessary stability condition (equivalent to \eqref{eq:alphamax}) coming from frequency analysis can be derived (Cor 3.3 from \cite{hale2001effects}):
\[
	\left|\frac{1 - c \tilde{g}}{1 + c \tilde{g}}\right| + 2 \left|\frac{c k_p}{1 + c \tilde{g}}\right| < 1.
\]
It is more restrictive and taking $k_p \neq 0$ practically leads to a decrease of the robustness (even if some other performances might be enhanced). This phenomenon has been studied in many articles \cite{helmicki1991ill,morgul1995stabilization,ATJdrilling}. Hence, to be robust to delays in the loop, one then needs to ensure the following:
\[
	0 \leq k_p \leq \frac{1}{2c} \left( |1 + c \tilde{g}| - |1 - c \tilde{g}| \right) = \tilde{g} = 2.1 \cdot 10^{-3}.
\]
This inequality on $k_p$ comes when considering the infinite dimensional problem and does not arise when dealing with any finite dimensional model. That point enlightens that it is more realistic to consider the infinite dimensional problem. For more information on that point, the interested reader can refer to \cite{barreauThese}.

\section{Practical Stability of system~\eqref{eq:problem}-\eqref{eq:controller}}

The experiments conducted previously shows for $\Omega_0$ large, the trajectory of the nonlinear system~\eqref{eq:problem}-\eqref{eq:controller} goes exponentially to $X_{\infty}$, so does the linear system. In other words, the previous result is a local stability test for the nonlinear system in case of large desired angular velocity $\Omega_0$ and does not extend straightforward to a global analysis. 

In general, for a nonlinear system, ensuring the global exponential stability of an equilibrium point could be complicated. In many engineering situations, global exponential stability is not the requirement. Indeed, considering uncertainties in the system or nonlinearities, it is far more reasonable and acceptable for engineers to ensure that the trajectory remains close to the equilibrium point. This property is called dissipativeness in the sense of Levinson \cite{levinson1944transformation} or practical stability (also called globally uniformly ultimately bounded in \cite{khalil1996nonlinear}).

\begin{definition} 
System~\eqref{eq:problem} is \textbf{practically stable} if there exists $X_{bound} \geq 0$ such that for $X$ the solution of \eqref{eq:problem} with initial condition $X_0 \in \mathcal{H}$:
\begin{equation} \label{eq:practical}
	\forall \eta > 0, \exists T_{\eta} > 0, \forall t \geq T_{\eta}, \quad \| X(t) - X_{\infty} \|_{\mathcal{H}} \leq X_{bound} + \eta.
\end{equation}
\end{definition}

Saying that system~\eqref{eq:problem}-\eqref{eq:controller} is practically stable means that there exists $T_{\eta} > 0$ such that for any $x \in [0, 1]$, $\phi_t(x,t)$ stays close to $\Omega_0$ for $t \geq T_{\eta}$. This property has already been applied to a drilling system in \cite{saldivar2016control} for instance. In the nonlinear case, the aim is then to design a control law reducing the amplitude of the stick-slip when it occurs. 

The objective of this section is to derive an LMI test ensuring the practical stability of nonlinear system~\eqref{eq:problem} together with the control defined in \eqref{eq:controller}:
\begin{equation} \label{eq:problemNonLin}
	\left\{
		\begin{array}{l}
			\phi_{tt}(x,t) = c^2 \phi_{xx}(x,t) - \gamma_t \phi_t(x,t), \\
			\phi_x(0,t) = (\tilde{g} + k_p) \phi_t(0,t) - k_p \Omega_0 + k_i C_2 Z(t), \\
			\phi_t(1,t) = C_1Z (t), \\
			\dot{Z}(t) = A Z(t) + B \left[ \begin{smallmatrix} \phi_t(0,t) \\ \phi_x(1,t) \end{smallmatrix} \right] + B_2 \left[ \begin{smallmatrix} \Omega_0 \\ T_{nl}(z_1(t)) \end{smallmatrix} \right],
		\end{array}
	\right.
\end{equation}

\begin{remark} To simplify the writing, $\phi$ can both means the solution of the linear or nonlinear system depending on the context. In this part, it refers to a solution of nonlinear system~\eqref{eq:problemNonLin}. \end{remark}

\subsection{Practical stability of system~\eqref{eq:problemNonLin}}

The idea behind practical stability is to embed the static nonlinearities by the use of suitable sector conditions, as it has been done for the saturation for instance \cite{tarbouriech2011stability}. Then, the use of robust tools will lead to some LMI tests ensuring the practical stability.

\begin{lemma} \label{lem:sectorsConditions} For almost all $\tilde{z}_1 \in \R$, the following holds:
\begin{equation} \label{eq:sectorsConditions}
	\begin{array}{ll}
		T_{nl}( \tilde{z}_1 + \Omega_0 )^2 \leq T_{max}^2, & T_{min}^2 \leq T_{nl}(\tilde{z}_1 + \Omega_0)^2, \\
		-2 \left( \tilde{z}_1 + \Omega_0 \right) T_{nl}( \tilde{z}_1 + \Omega_0 ) \leq 0.
	\end{array}
\end{equation}
\end{lemma}
\begin{proof}
These inequalities can be easily verified using \eqref{eq:torque}.
\end{proof}

These new information are the basis of the following theorem.

\begin{theo} \label{theo:practical} Let $N \in \mathbb{N}$ and $V_{max} > 0$. If there exist $P_N \in \mathbb{S}^{2+2(N+1)}$, $R = \diag(R_1, R_2) \succeq 0, S = \diag(S_1, S_2) \succ 0, Q \in \mathbb{S}^2_+, \tau_0, \tau_1, \tau_2, \tau_3 \geq 0$ such that \eqref{eq:extraCondition} holds together with:
\begin{equation} \label{eq:LMIpractical}
	\begin{array}{l}
		\Xi_N = \bar{\Theta}_{N}  - \tau_0 \Pi_0 - \tau_1 \Pi_1 - \tau_2 \Pi_2 - \tau_3 \Pi_3 \prec 0, \\
		P_N + S_N \succ 0,
	\end{array}
\end{equation}
where 
\[
	\begin{array}{cl}
		\!\! \bar{\Theta}_{N} \!\!\!\!& = \diag( \Theta_N, 0_2 )  \\
		& \hfill- \alpha_2 \He\left( (F_{m1} - T_0 F_{m2} )^{\top} e_1^{\top} P_N \tilde{F}_N \right), \\
		\!\! \tilde{F}_N \!\!\!\!& = \left[ F_N \ 0_{2(N+1)+2, 2} \right], \quad e_1 = \left[ 1 \ 0_{1, 2(N+1)+1} \right]^{\top}\!\!\!, \\
		\!\! F_{m1} \!\!\!\!& =  \left[ 0_{1, 2(N+1) + 4} \ 1 \ 0 \right], \quad F_{m2} = \left[ 0_{1, 2(N+1) + 4} \ 0 \ 1 \right], \\
		\!\! \Pi_0 \!\!\!\!& = V_{max} F_{m2}^{\top} F_{m2} - \tilde{F}_N^{\top} ( P_N + S_N) \tilde{F}_N, \\
		\!\! \Pi_1 \!\!\!\!& = \pi_2^{\top} \pi_2 - T_{max}^2 \pi_3^{\top} \pi_3, \quad \Pi_2 = T_{min}^2 \pi_3^{\top} \pi_3 - \pi_2^{\top} \pi_2, \\
		\!\! \Pi_3 \!\!\!\!& = - \He( (\pi_1 + \Omega_0 \pi_3)^{\top} \pi_2 ), \\
		\!\! \pi_1 \!\!\!\!& = \left[ 1, 0_{1, 2(N+1)+3}, 0, 0 \right], \quad \pi_2 = \left[ 0_{1, 2(N+1)+4}, 1, 0 \right], \\
		\!\! \pi_3 \!\!\!\!& = \left[ 0_{1, 2(N+1)+4}, 0, 1 \right], \\
	\end{array}
\]
and all the parameters defined as in Theorem~\ref{theo:StabLin}, then the equilibrium point $X_{\infty}$ of system~\eqref{eq:problemNonLin} is practically stable. More precisely, equation \eqref{eq:practical} holds for $X_{bound} = \sqrt{V_{max} \varepsilon_1^{-1}}$ where $\varepsilon_1$ is defined in \eqref{eq:epsilon1}. \end{theo}


\begin{proof}
First, let do the same change of variable as before:
\[
	\tilde{X} = X - X_{\infty}.
\]
Since the nonlinearity affects only the ODE part of system~\eqref{eq:problemNonLin}, the difference with the previous part lies in the dynamic of $\tilde{Z}$:
\[
	\begin{array}{rl}
		\dot{\tilde{Z}}(t) \!\!\!\!& = \frac{d}{dt} \left( Z(t) - \left[ \begin{smallmatrix} z_1^{\infty} \\ z_2^{\infty} \end{smallmatrix} \right] \right) \\
		&= A \tilde{Z}(t) + B \left[ \begin{smallmatrix} \tilde{\phi}_t(0,t) \\ \tilde{\phi}_x(1,t) \end{smallmatrix} \right] + B_2 \left[ \begin{smallmatrix} 0 \\ T_{nl}(\tilde{z}_1(t) + \Omega_0) - T_0 \end{smallmatrix} \right].
	\end{array}
\]

Using the same Lyapunov functional as in \eqref{eq:lyap}, the positivity is ensured in the exact same way. For the bound on the time derivative, following the same strategy as before, we easily get:
\begin{equation} \label{eq:lyapDot}
	\begin{array}{rl}
		\!\!\!\!\!\!\!\!\dot{V}_N(\tilde{X}) \!\!\!\!& \leq \tilde{\xi}_N^{\top} \Theta_N \tilde{\xi}_N - 2 \alpha_2 (T_{nl}(\tilde{z}_1 + \Omega_0) - T_0) e_1^{\top} P_N \tilde{Z}_N \\
		&\displaystyle \hfill + \sum_{k=0}^N (2k+1) \tilde{\Chi}_k^{\top} Q \tilde{\Chi}_k - \int_0^1 \tilde{\chi}^{\top}(x) Q \tilde{\chi}(x) dx.
	\end{array}
\end{equation}
We introduce a new extended state variable $\bar{\xi}_N = \left[ \begin{matrix} \tilde{\xi}_N^{\top} & T_{nl}(\tilde{z}_1 + \Omega_0) & 1 \end{matrix} \right]^{\top}$.
Using the notation of Theorem~\ref{theo:practical}, equation~\eqref{eq:lyapDot} rewrites as:
\begin{multline*}
	\dot{V}_N(\tilde{X}) \ \leq \ \bar{\xi}_N^{\top} \bar{\Theta}_N \bar{\xi}_N^{\top} + \sum_{k=0}^N (2k+1) \tilde{\Chi}_k^{\top} Q \tilde{\Chi}_k \\
	- \int_0^1 \tilde{\chi}^{\top}(x) Q \tilde{\chi}(x) dx.
\end{multline*}
It is impossible to ensure $\bar{\Theta}_N \prec 0$ since its last $2$ by $2$ diagonal block is $0_2$. 

We then use the definition of practical stability. We want to show that if there exists $\varepsilon_3 > 0$ such that $\dot{V}_N(\tilde{X}) \leq - \varepsilon_3 \| \tilde{X}\|_{\mathcal{H}}^2$ when $V_N(\tilde{X}) \geq V_{max}$ then the system is practically stable. \\
Let $\mathcal{S} = \left\{ \tilde{X} \in \mathcal{H} \ | \ V_N(\tilde{X}) \leq V_{max} \right\}$, the previous assertion implies that this set is invariant and attractive. 
Using \eqref{eq:epsilon1}, we get that $\left\{ \tilde{X} \in \mathcal{H} \ | \ \| \tilde{X} \|_{\mathcal{H}} \leq X_{bound} = V_{max}^{1/2} \varepsilon_1^{-1/2} \right\} \supseteq \mathcal{S}$, meaning that the system is practically stable.
\\
A sufficient condition to be practically stable is that $V_N$ should be strictly decreasing when outside the ball of radius $V_{max}$. This condition rewrites as $V_N(\tilde{X}) \geq V_{max}$ and the following holds:
\[
	\begin{array}{rl}
		\!\!\!\!V_N(\tilde{X}) - V_{max} \!\!\!\!&\displaystyle \geq \tilde{Z}_N^{\top} P_N \tilde{Z}_N + \int_0^1 \!\! \tilde{\chi}^{\top}(x) S \tilde{\chi}(x) dx - V_{max}\\
		& \displaystyle \geq - \bar{\xi}_N^{\top} \Pi_0 \bar{\xi}_N \geq 0.
	\end{array}
\]
The previous inequality is obtained using Bessel inequality~\eqref{eq:bessel} on $\int_0^1 \tilde{\chi}^{\top}(x) S \tilde{\chi}(x) dx$. Hence, $V_N(\tilde{X}) \geq V_{max}$ if $\bar{\xi}_N^{\top} \Pi_0 \bar{\xi}_N \leq 0$.

Noting that $\tilde{z}_1 = \pi_1 \bar{\xi}_N, T_{nl}(\tilde{z}_1 + \Omega_0) = \pi_2 \bar{\xi}_N$ and $1 = \pi_3 \bar{\xi}_N$, Lemma~\ref{lem:sectorsConditions} rewrites as:
\[
	\forall i \in \{ 1, 2, 3 \}, \quad \bar{\xi}_N \ \Pi_i \ \bar{\xi}_N \leq 0.
\]
Consequently, a sufficient condition to be practically stable is:
\begin{equation} \label{eq:LMI1}
	\forall \bar{\xi}_N \neq 0 \text{ s. t. } \forall i \in [0, 3], \bar{\xi}_N^{\top} \Pi_i \bar{\xi}_N \leq 0, \quad \bar{\xi}_N^{\top} \bar{\Theta}_N \bar{\xi}_N < 0.
\end{equation}
The technique called \textit{S-variable}, explained in \cite{Svariable} for instance, translates the previous inequalities into an LMI condition. Indeed, Theorem 1.1 from \cite{Svariable} shows that condition \eqref{eq:LMI1} is verified if there exists $\tau_0, \tau_1, \tau_2, \tau_3 > 0$ such that 
\[
	\bar{\Theta}_N - \sum_{k=0}^3 \tau_i \Pi_i \prec 0.
\]
Consequently, in a similar way than for Theorem~\ref{theo:StabLin}, condition~\eqref{eq:LMIpractical} implies that $\mathcal{S}$ is an invariant and attractive set for system~\eqref{eq:problemNonLin}. Then, the equilibrium point $X_{\infty}$ of system~\eqref{eq:problemNonLin} is practically stable with $X_{bound} = \sqrt{V_{max} \varepsilon_1^{-1}}$.
\end{proof}


Note that if the torque function is not perfectly known, one can change the lower and upper bounds $T_{min}$ and $T_{max}$ to get a more conservative result but robust to uncertainties on $T_{nl}$.
For \eqref{eq:LMIpractical} to be feasible, the constraint $\bar{\xi}_N^{\top} \Pi_1 \bar{\xi}_N \leq 0$ must hold, meaning that an upper bound of $T_{nl}$ needs to be proposed.


\subsection{On the optimization of $X_{bound}$}

The condition~\eqref{eq:LMIpractical} is a Bilinear Matrix Inequality (BMI) since $\tau_0$, $\tau_1$, $\tau_2$ and $\tau_3$ are decision variables and it is therefore difficult to get its global optimum. Nevertheless, the following lemma gives a sufficient condition for the existence of a solution to this problem.

\begin{coro} \label{coro:practical} There exists $V_{max}$ and $\tau_0 > 0$ such that Theorem~\ref{theo:practical} holds if and only if there exists $N > 0$ such that LMIs \eqref{eq:extraCondition} and \eqref{eq:LMI} are satisfied.

In other words, the equilibrium point of system~\eqref{eq:problemNonLin} is practically stable if and only if the linear system~\eqref{eq:problemLin} is exponentially stable.
\end{coro}

\begin{proof} Note first that expending \eqref{eq:LMIpractical} with $\tau_3 = 0$ leads to:
\begin{equation} \label{eq:extendedLMI}
	{\small \Xi_N\! = \!\left[ \begin{array}{c|c|c} 
		\Theta_{N,\tau_0/2} & 
		\multicolumn{2}{|c}{\kappa_{P_N}} \\
		\hline
		\multirow{3}{*}{$\kappa_{P_N}^{\top}$} & \tau_2 - \tau_1 & 0 \\
		\cline{2-3}
		& \multirow{2}{*}{0} &  \tau_1 T_{max}^2 - \tau_2 T_{min}^2 \\
		& & - \tau_0 V_{max}
	\end{array} \right],}
\end{equation}
where $\kappa_{P_N} \in \R^{(4+2(N+1)) \times 2}$ depends only on $P_N$. 

\textit{Proof of sufficiency:} Assume there exists $N > 0$ such that LMIs \eqref{eq:extraCondition} and \eqref{eq:LMI} are satisfied. Considering $\tau_2 = 0$ and using Schur complement on $\Xi_N$, $\Xi_N \prec 0$ is equivalent to:
\begin{equation*}
	 \Theta_{N,\tau_0/2}
	 - \kappa_{P_N} \left[ \begin{matrix} -\frac{1}{\tau_1} & 0 \\ 0 & \frac{1}{\tau_1 T_{max}^2 - \tau_{0} V_{max}} \end{matrix} \right] \kappa_{P_N}^{\top} \prec 0,
\end{equation*}
with $\tau_{0} V_{max} > \tau_1 T_{max}^2$. Since $\Theta_N \prec 0$, considering $\tau_0$ small enough, $\tau_1$ large and $V_{max} > \tau_1 \tau_0^{-1} T_{max}^2$ the previous condition is always satisfied and Theorem~\ref{theo:practical} applies.

\textit{Proof of necessity:} Assume $\Xi_N \prec 0$ and \eqref{eq:extraCondition} holds. Then its first diagonal block must be definite negative. Consequently $\Theta_{N} \prec 0$ and, according to Theorem~\ref{theo:StabLin}, system~\eqref{eq:problemLin} is exponentially stable.
\end{proof}

\begin{remark} Note that \eqref{eq:extendedLMI} provides a necessary condition for Theorem~\ref{theo:practical} which is $\Theta_{N, \tau_0/2} \prec 0$. In other words, $\tau_0$ is related to the decay-rate of the linear system and we get the following condition: $\tau_0 < 2 \mu_{max}$.\end{remark}

Thanks to Corollary~\ref{coro:practical}, the following method should help solving the BMI if Theorem~\ref{theo:StabLin} is verified. Assuming that equations \eqref{eq:LMI} and \eqref{eq:extraCondition} are verified for a given $N \in \mathbb{N}$.
\begin{enumerate}
	\item Fix $\tau_0 = 2 \mu_{max}$ as defined in \eqref{eq:alphamax}.
	\item Check that equations \eqref{eq:LMI}, \eqref{eq:extraCondition} and \eqref{eq:LMIpractical} are satisfied for $V_{max}$ a strictly positive decision variable. If this is not the case, then decrease $\tau_0$ and do this step again. 
	\item Thanks to Corollary~\ref{coro:practical}, there exists a $\tau_0$ small enough for which equations \eqref{eq:LMI}, \eqref{eq:extraCondition} and \eqref{eq:LMIpractical} hold. Freeze this value.
	\item Since the problem is unbounded, it is possible to fix a variable without loss of generality, let $V_{max} = 10^4$ for instance and solve the following optimization problem:
	\[
		\hspace*{-0.6cm}\begin{array}{cl}
			\displaystyle \min_{P_N, S, R, Q, \tau_1, \tau_2, \tau_3, \varepsilon_P} & - \varepsilon_P \\
			\text{subject to} & \begin{array}{l}
				\eqref{eq:LMI}, \eqref{eq:extraCondition} \text{ and } \eqref{eq:LMIpractical}, \\
				P_N + S_N - \diag(\varepsilon_P, 0_{2N+3}) \succ 0, \\
				R \succeq 0, S \succ 0, Q \succ 0.
			\end{array}				
		\end{array}
	\]
	\item Then compute $X_{bound} = \sqrt{V_{max} \varepsilon_1^{-1}}$ where $\varepsilon_1$ is defined in \eqref{eq:epsilon1}.
\end{enumerate}

\section{Examples and Simulations}

The section is devoted to numerical simulations\footnote{Numerical simulations are done using a first order approximation with at least $80$ space-discretization points and $9949$ time-discretization points. Simulations are done using Yalmip \cite{1393890} together with the SDP solver SDPT-3 \cite{toh1999sdpt3}. The code is available at {\textit{https://homepages.laas.fr/mbarreau}}.} and draw some conclusions about the PI regulation. In the first subsection, we focus on the linear system and the second one is dedicated to the nonlinear case.

\subsection{On the linear model}

This subsection recaps the result on the linear model.

\subsubsection{Estimation of the decay-rate}

The main result is a direct application of Corollary~\ref{coro:alpha} for $k_p = 10^{-3}$ and $k_i = 10$. Indeed, using a dichotomy-kind algorithm, Table~\ref{tab:alpha} is obtained. It shows the estimated decay-rate $\mu$ at a given order between $0$ and $6$.

The first thing to note in Table~\ref{tab:alpha} is the hierarchy property, the decay-rate is an increasing function of the degree, as noted in Remark~\ref{rem:hierarchy}. Note also that the gap between order $0$ and order $1$ is significant, showing that using projections indeed improves the results. 

For orders higher than $2$, the estimated decay-rate increases slightly, and, up to a four digits precision, it reaches its maximum value at $N = 6$. Since it is then the maximum allowable decay-rate obtained using equation~\eqref{eq:alphamax}, that tends to show that the Lyapunov functional used in this paper together with condition~\eqref{eq:thetaMu} are sharp and provide a good analysis. Figure~\ref{fig:energyLin} represents a simulation on the linear system and it confirms the same observation. Indeed, one can see that the energy of the system is well-bounded by an exponential curve and that the bound becomes more and more accurate as $N$ increases.

\begin{table}
	\centering
	\begin{tabular}{c|ccccc}
		Order & $N = 0$\! & $N=1$\! & $N=2$\! & $N=3$\! & $N = 6$\! \\
		\hline
		$\mu$ $(\times 10^{-3})$ & $0.87$ & $4.24$ & $7.31$ & $7.59$ & $7.73$
	\end{tabular}
	\vspace{0.5cm}
	\caption{Estimated decay-rate as a function of the order $N$ used. Note that $\mu_{max} = 7.73 \cdot 10^{-3}$ is calculated using \eqref{eq:alphamax} for $k_p = 10^{-3}$, $k_i = 10$.}
	\label{tab:alpha}
\end{table}

\begin{figure}
	\centering
	\includegraphics[width=8.3cm]{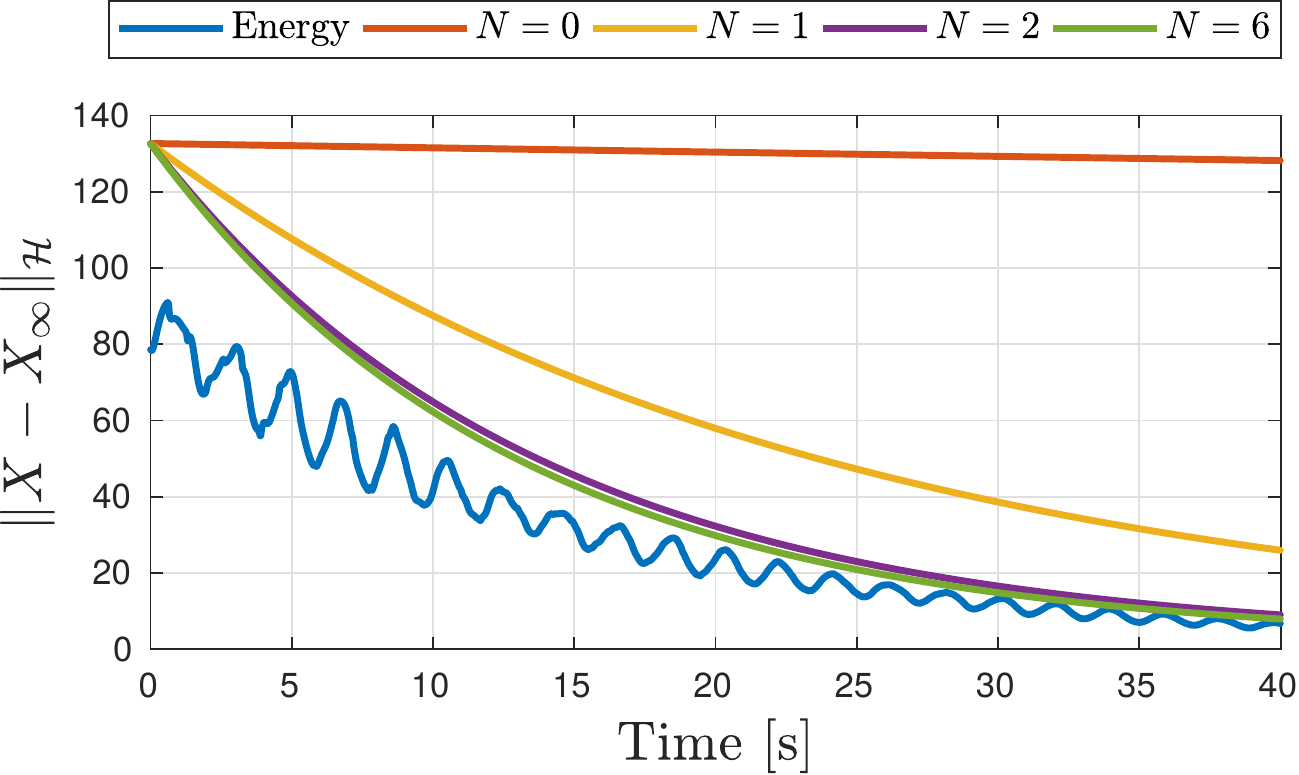}
	\caption{Energy of $X$ with the linear system for $k_p = 10^{-3}$, $k_i = 10$ and $\Omega_0 = 5$. The initial condition is $\phi^0(x) = 4 \left( \int_0^x \phi_x^{\infty}(s) ds + 0.1 \cos(2x) \right)$, $\phi^1 = 2 \Omega_0$ and $Z(0) = 2 \left[ z_1^{\infty} \ z_2^{\infty} \right]^{\top}$.}
	\label{fig:energyLin}
\end{figure} 

\subsubsection{Stability of the closed loop system}

We are interested now in estimating the stability area in terms of the gains $k_p$ and $k_i$ such that the decay-rate of the coupled system is $\mu_{max}$ for an order $N = 5$. That leads to Figure~\ref{fig:kpki} where it is easy to see that increasing the gain $k_p$ decreases the range of possible $k_i$ while it increases its speed (see equation~\eqref{eq:alphamax}). It is quite natural to note that the larger $k_i$ is, the slower the system is, while increasing the proportional gain leads to a faster system. As a conclusion, for small values of $k_p$ and $k_i$, the system is stable and that was the conclusion of the two papers \cite{ATJdrilling,ATJECC} using a different Lyapunov functional. Note that with the previous papers, it was not possible to quantify the notion of ``small enough gains $k_p$ and $k_i$'' while it is possible to give an estimation with the method of this paper.

\begin{figure}
	\centering
	\includegraphics[width=7.5cm]{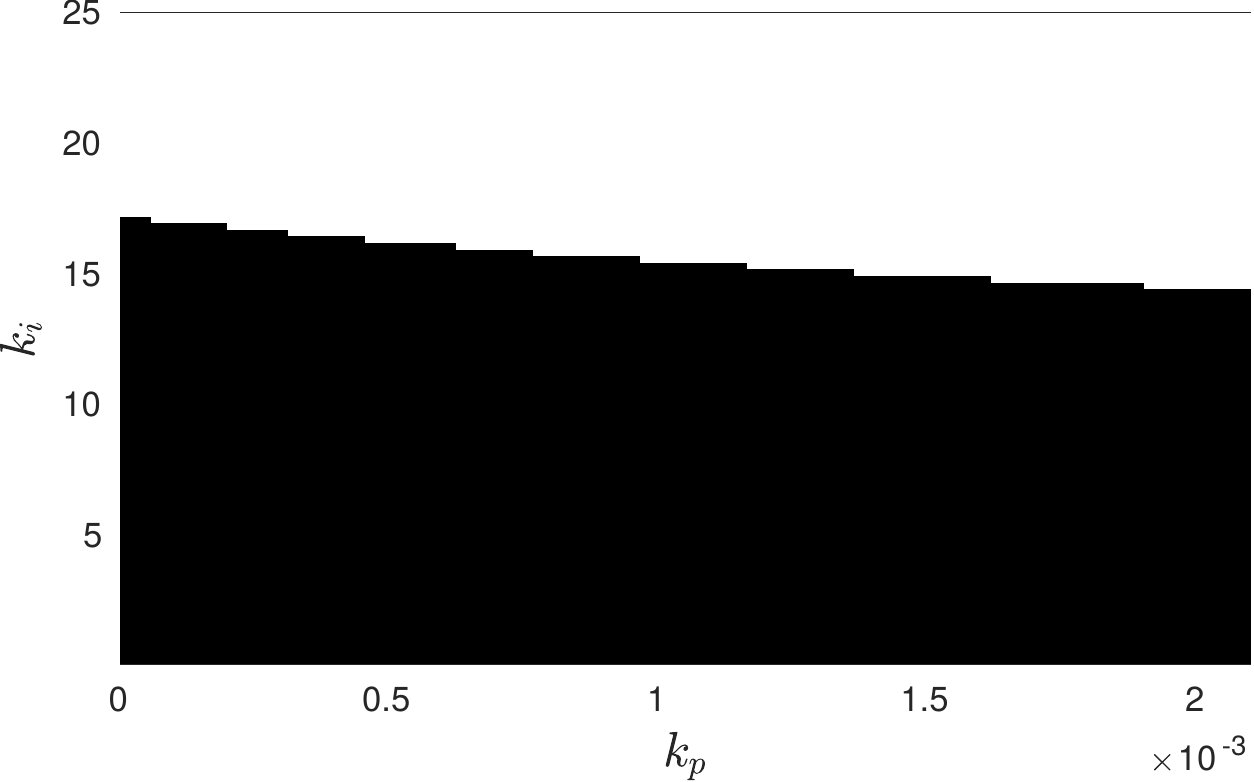}
	\caption{Values of gains $k_p$ and $k_i$ leading to a stable system with the maximum decay-rate for $N = 5$ using Theorem~1. The black area is stable and the white area is said unstable up to an order $5$.}
	\label{fig:kpki}
\end{figure} 

\subsection{The stick-slip effect on the nonlinear model}

The previous subsection shows that the linear model is globally asymptotically stable for some values of gains $k_p$ and $k_i$, so the nonlinear system~\eqref{eq:problemNonLin} is locally asymptotically stable for a large desired angular velocity $\Omega_0$. This can been verified in Figure~\ref{fig:bassinSimu} for $k_p = 10^{-3}$, $k_i = 10$ and $\Omega_0 = 10$. We can see that the linear and the nonlinear systems behave similarly if their initial condition is close to the equilibrium.

\begin{figure}
	\centering
	\subfloat[Systems~\eqref{eq:problemLin} and \eqref{eq:problemNonLin}: $\Omega_0 = \phi^1 = 10$, $\phi^0(x) = \left( 1 + 0.32 \sin(x) \right) \int_0^x \phi_x^{\infty}(s) ds$ and $Z(0) = \left( z_1^{\infty} \ z_2^{\infty} \right)^{\top}$.]{\includegraphics[width=7.5cm]{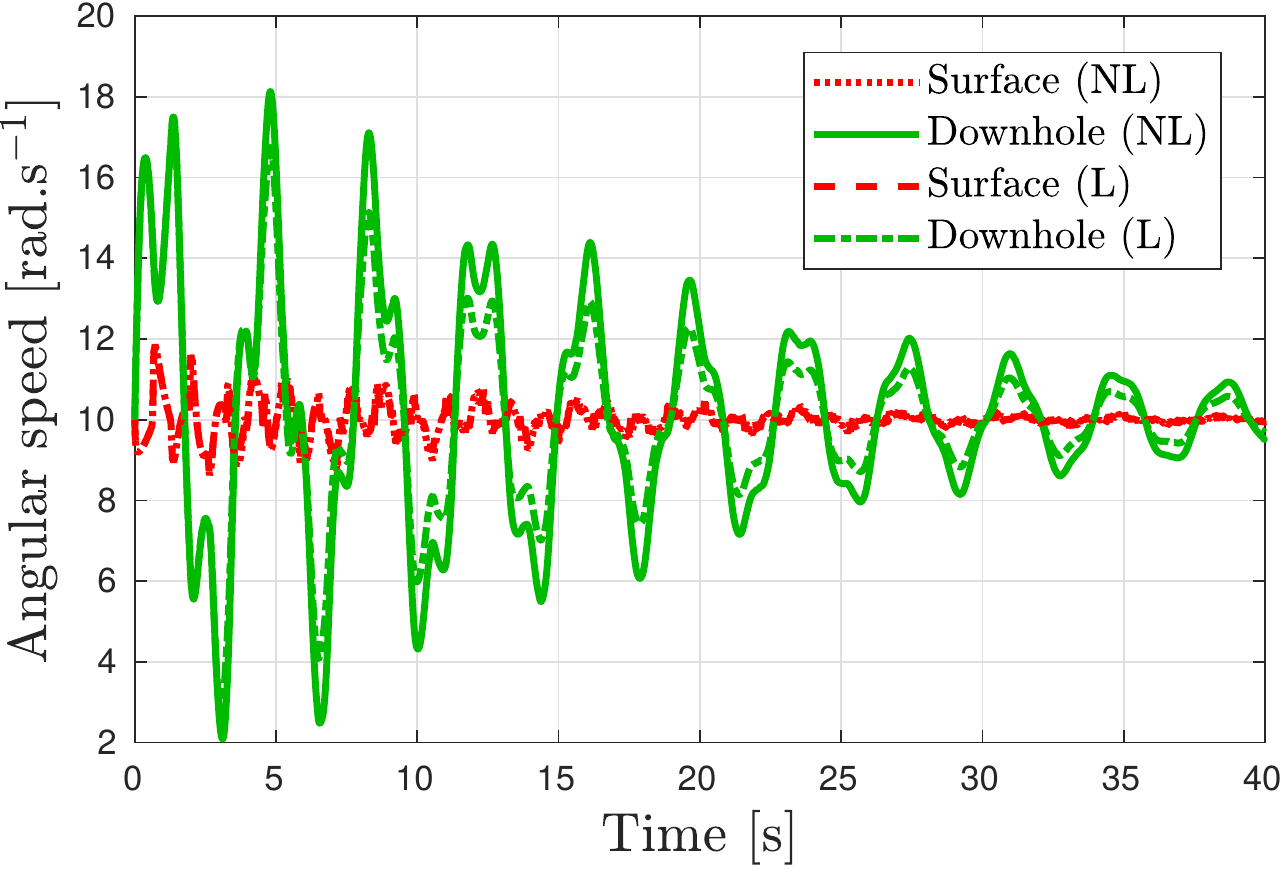} \label{fig:bassinSimu}} \\
	\subfloat[System~\eqref{eq:problemNonLin}: $\Omega_0 = 20$, $\phi^0 = 0$, $\phi^1 = 0$ and $Z(0) = 0_{2,1}$.]{\includegraphics[width=7.5cm]{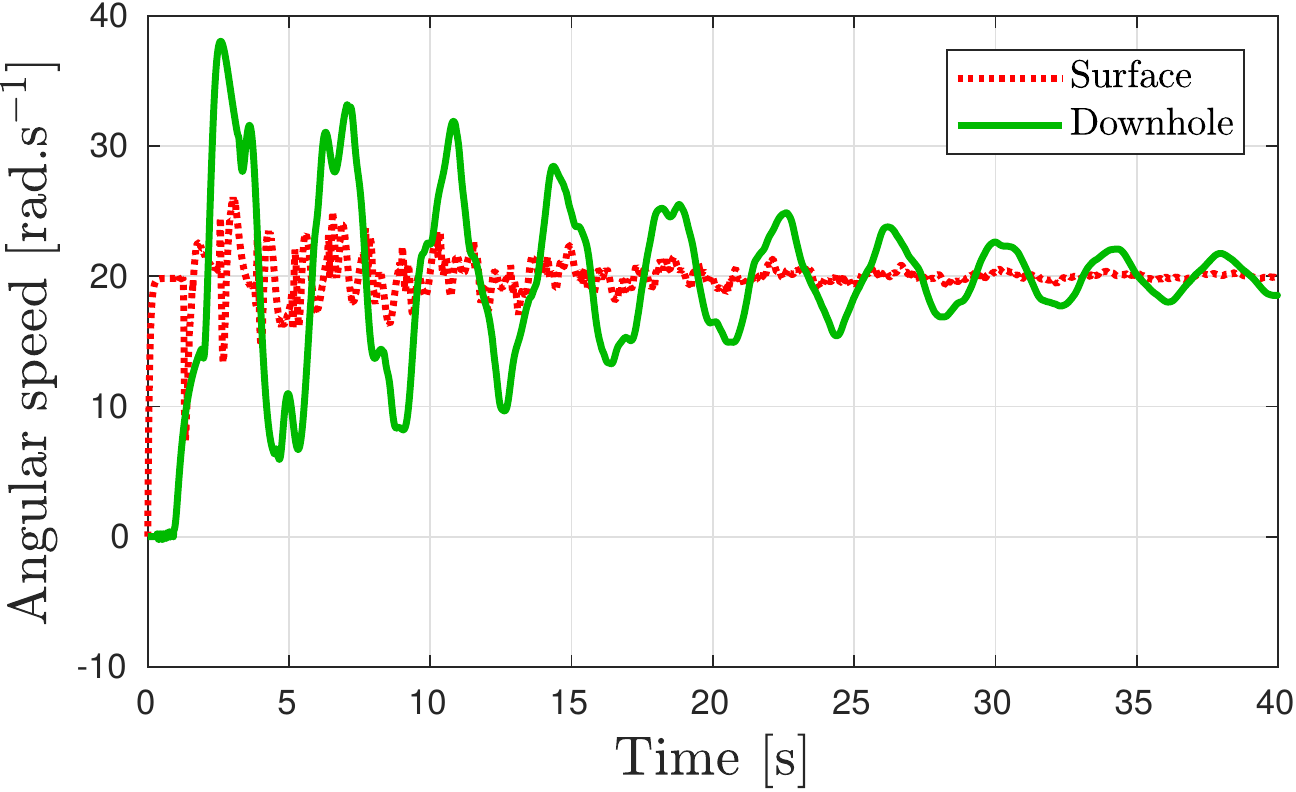} \label{fig:highOmega0}} \\
	\subfloat[System~\eqref{eq:problemNonLin}: $\Omega_0  = \phi^1 = 5$, $\phi^0(x) = \left( 1 + 0.1 \sin(x) \right) \int_0^x \phi_x^{\infty}(s) ds$ and $Z(0) = \left( z_1^{\infty} \ z_2^{\infty} \right)^{\top}$.]{\includegraphics[width=7.5cm]{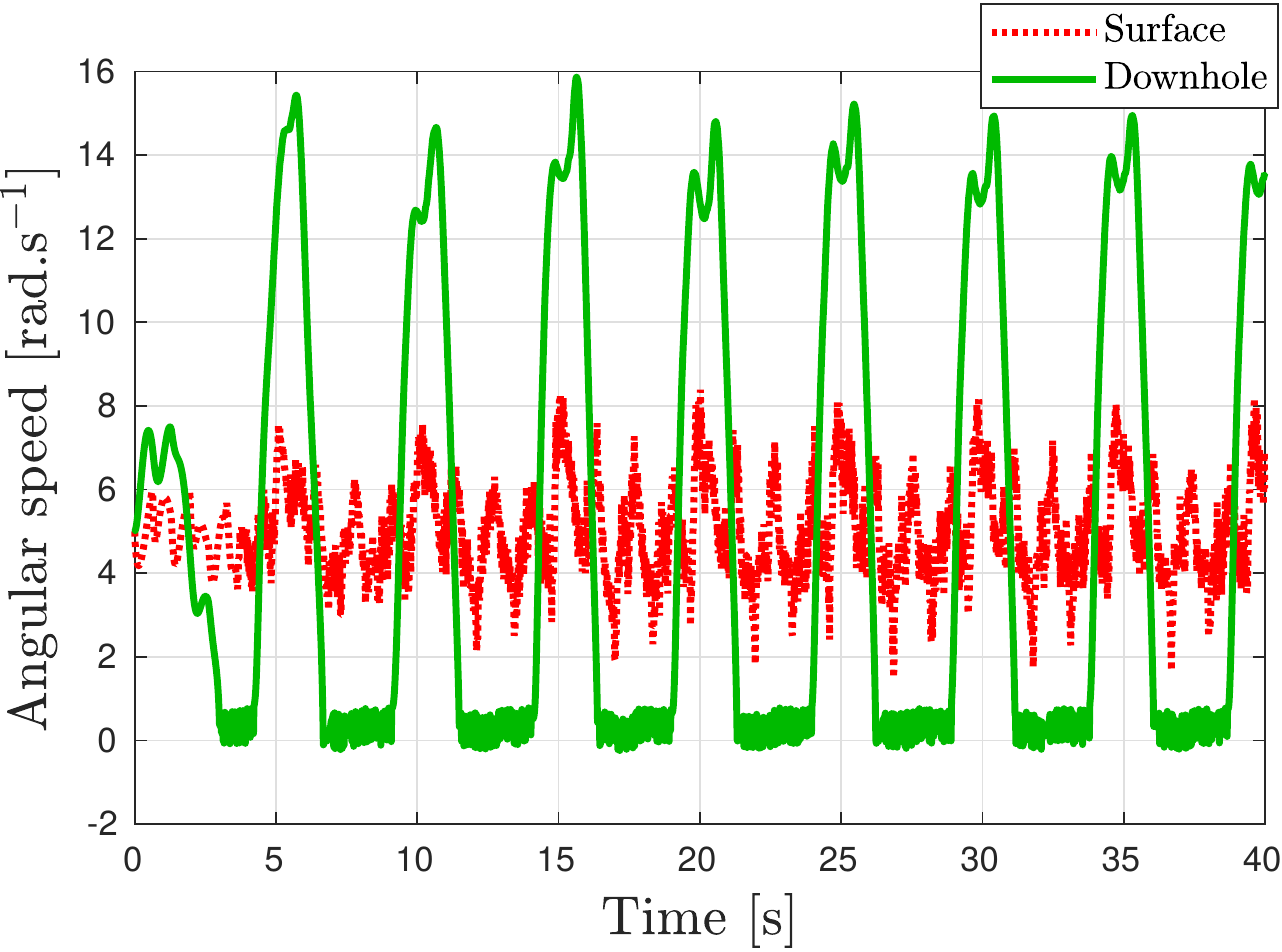} \label{fig:smallOmega0}}
	\caption{Numerical simulations for systems~\eqref{eq:problemLin} and \eqref{eq:problemNonLin} with $k_p = 10^{-3}$, $k_i = 10$.}
	\label{fig:numericalSimulations}
\end{figure} 

The higher $\Omega_0$ is, the larger the basin of attraction is. Indeed, the regulation tries to bring the system into the ``quasi''-linear zone of $T_{nl}$, where it is close to a constant $T_{min}$ as seen in Figure~\ref{fig:torque} and consequently, the stick-slip phenomenon may occur at the beginning but it is not effective for a long time as depicted in Figure~\ref{fig:highOmega0} which is a numerical simulation on the nonlinear model with a zero initial condition.


The real challenge is then the case of low desired angular velocities $\Omega_0$. The result of a simulation on the nonlinear model for $\Omega_0 = 5$ rad.s${}^{-1}$ ($k_p = 10^{-3}$, $k_i = 10$) is displayed in Figure~\ref{fig:smallOmega0}. First of all, note that the oscillations are with a frequency of $0.2$Hz and with an amplitude of roughly $15$ rad.s${}^{-1}$. This is very close to what has been estimated using Figure~\ref{fig:stickslip}. Then, the model presented in Section~\ref{sec:model} seems to be a valid approximation of the real behavior, at least concerning the stick-slip phenomenon.


\subsection{Practical stability analysis}

Now, one can evaluate the amplitude of the oscillations using Theorem~\ref{theo:practical} with $k_p = 10^{-3}$ and $k_i = 10$. The result for several values of $\tau_0$ and for an order between $0$ and $7$ is depicted in Figure~\ref{fig:xboundOrder}. Note that after order $8$, there are numerical errors in the optimization process and the result is not accurate. The maximum $\tau_0$ is $2 \mu_{max} = 0.0155$ and the higher $\tau_0$ is, the better is the optimization. It appears that for $k_i = 10$, $k_p = 10^{-3}$ and $\Omega_0 = 5$, the optimal $X_{bound}$ is around $28$.

Figure~\ref{fig:energy} shows the energy of the system as a function of time. One can see that the bound $X_{bound}$ is quite accurate since the error between the maximum of the auto-oscillations and $X_{bound}$ is around $53\%$. Moreover, note that $\max | z_1 - \Omega_0 | = 11.7 = 0.4 X_{bound}$, in other words, nearly half of the oscillations are concentrated in the variable $z_1$, which means the stick-slip mostly acts on the variable $z_1$ and does not affect much the rest of the system. Particularly, it seems very difficult to estimate the variation of $z_1$ knowing only $\phi_t(0,t)$.

The final observation is about the variation of $X_{bound}$ for different $\Omega_0$. This is depicted in Figure~\ref{fig:xboundOmega0}. Up to errors in the numerical optimization, it seems that there is not important variation of $X_{bound}$ when $\Omega_0$ increases. This is counter-intuitive and does not reflect the observations made with Figure~\ref{fig:numericalSimulations}. One explanation is that we didn't state that $T_{nl}$ is a strictly decreasing function for positive $\theta$. 

\begin{figure}
	\centering
	\includegraphics[width=8.3cm]{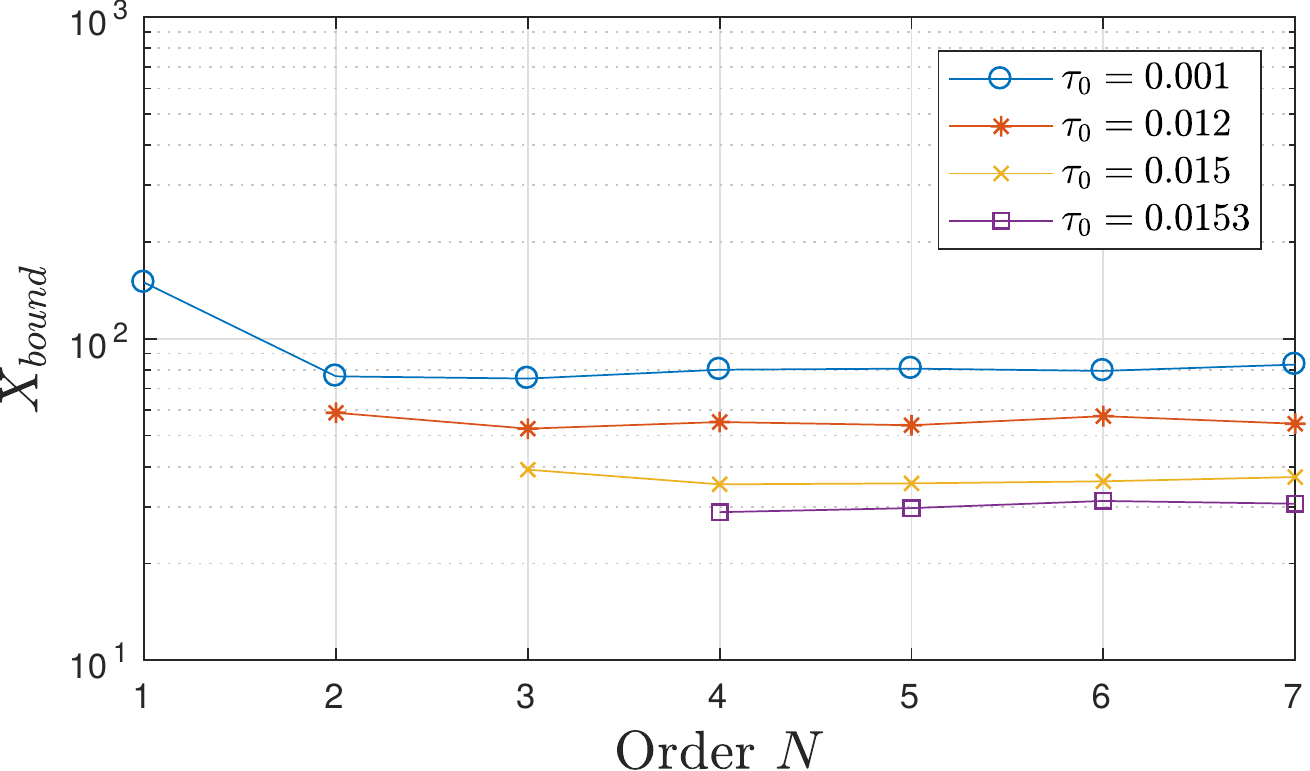}
	\caption{Solution $X_{bound}$ of BMI \eqref{eq:LMIpractical} for $k_p = 10^{-3}$, $k_i =  10$ and $\Omega_0 = 5$ and some values of $\tau_0$. The limit is $X_{bound} = 28$.}
	\label{fig:xboundOrder}
\end{figure} 

\begin{figure}
	\centering
	\includegraphics[width=8.3cm]{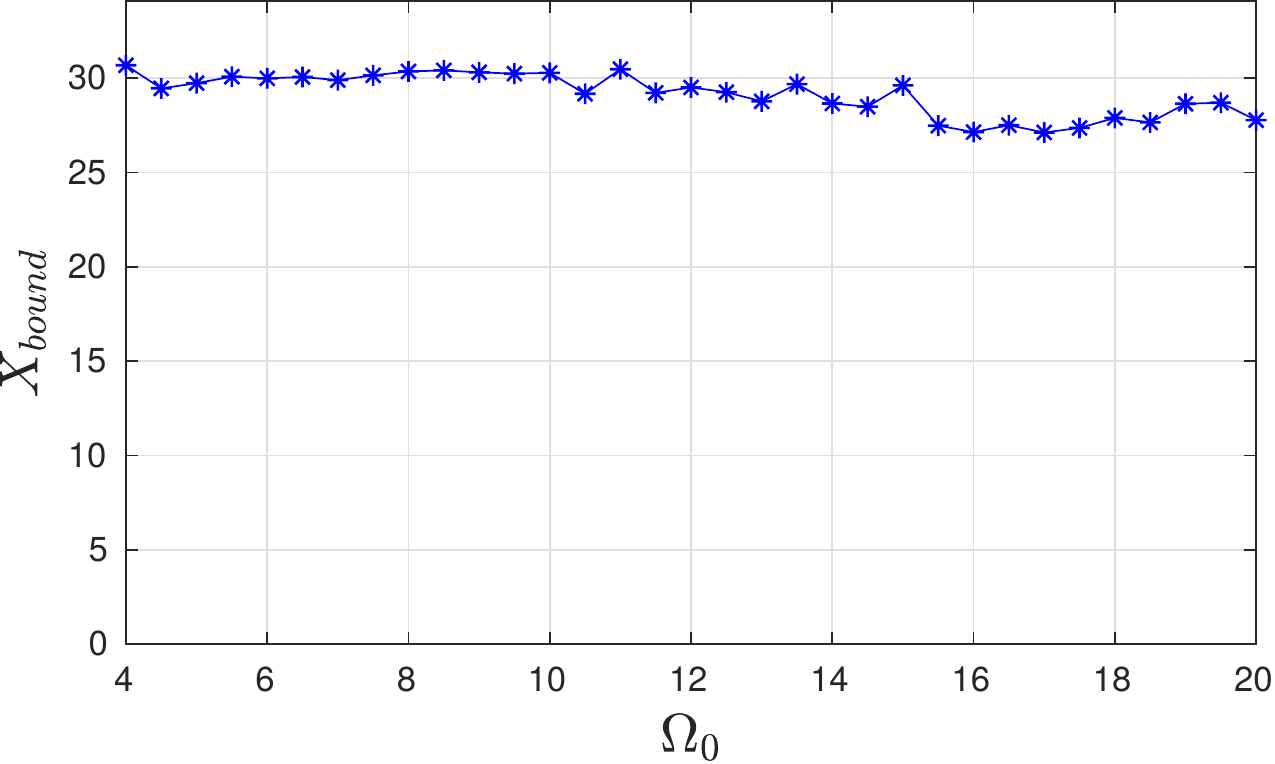}
	\caption{Solution $X_{bound}$ of BMI \eqref{eq:LMIpractical} for $k_p = 10^{-3}$, $k_i =  10$, $\tau_0 = 0.0153$ and $N = 5$.}
	\label{fig:xboundOmega0}
\end{figure} 

\begin{figure}
	\centering
	\includegraphics[width=8.3cm]{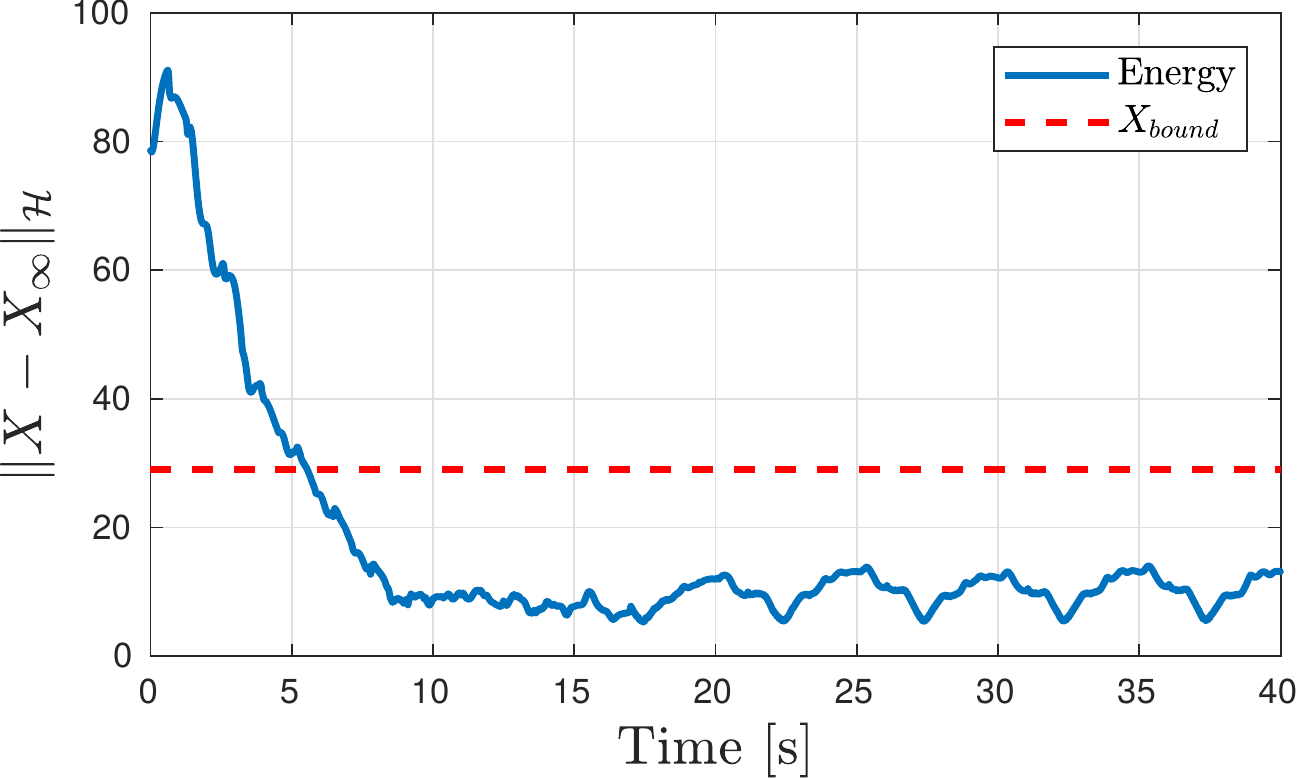}
	\caption{Energy of $X$ with the nonlinear system for $k_p = 10^{-3}$, $k_i = 10$ and $\Omega_0 = 5$. The initial condition is the same than in Figure~\ref{fig:energyLin}.}
	\label{fig:energy}
\end{figure} 

\subsection{Design of a PI controller}

Finally, the problem stated at the beginning of the paper was to find the best PI controller, meaning that it minimizes $X_{bound}$. The plot in Figure~\ref{fig:bestKi} shows that the value of the integral gain $k_i$ does impact the oscillations due to stick-slip since $X_{bound}$ increases from $25$ to $43$ for $k_i \in [0.5, 16]$ and the LMIs become infeasible after this point (these values have been obtained with $\Omega_0 = 5$, $k_p = 10^{-3}$ and $N = 5$). 

To stay robustly stable against small delay in the loop, as noted in Section~\ref{sec:kp}, we should consider $0 \leq k_p \leq 2.1 \cdot 10^{-3}$ which does not offer a large set of choices. For $k_i = 10$, we get Figure~\ref{fig:bestKp}. It appears that the minimum of $X_{bound}$ is obtained for $k_p$ close to $0$. Consequently, increasing the gain $k_p$ seems not to reduce the stick-slip effect. Consequently, even if a PI controller does not weaken the stick-slip effect, the equilibrium point of the controlled system is practically stable. Moreover, it does enable an oscillation around the desired equilibrium point and a local convergence to that point. 

\begin{figure}
	\centering
	\subfloat[$X_{bound}$ for $k_p = 10^{-3}$ but with different values of $k_i$]{\includegraphics[width=8.3cm]{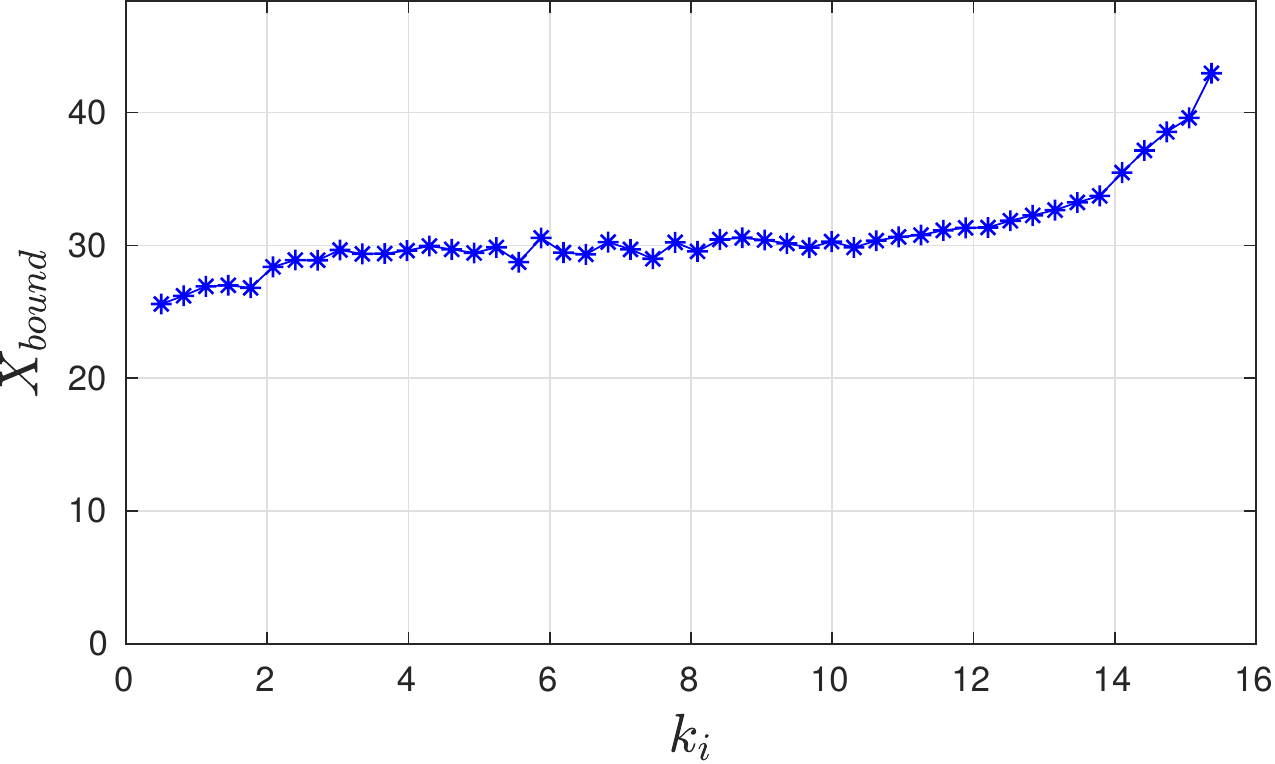} \label{fig:bestKi}}\\
	\subfloat[$X_{bound}$ for $k_i = 10$ but with different values of $k_p$]{\includegraphics[width=8.3cm]{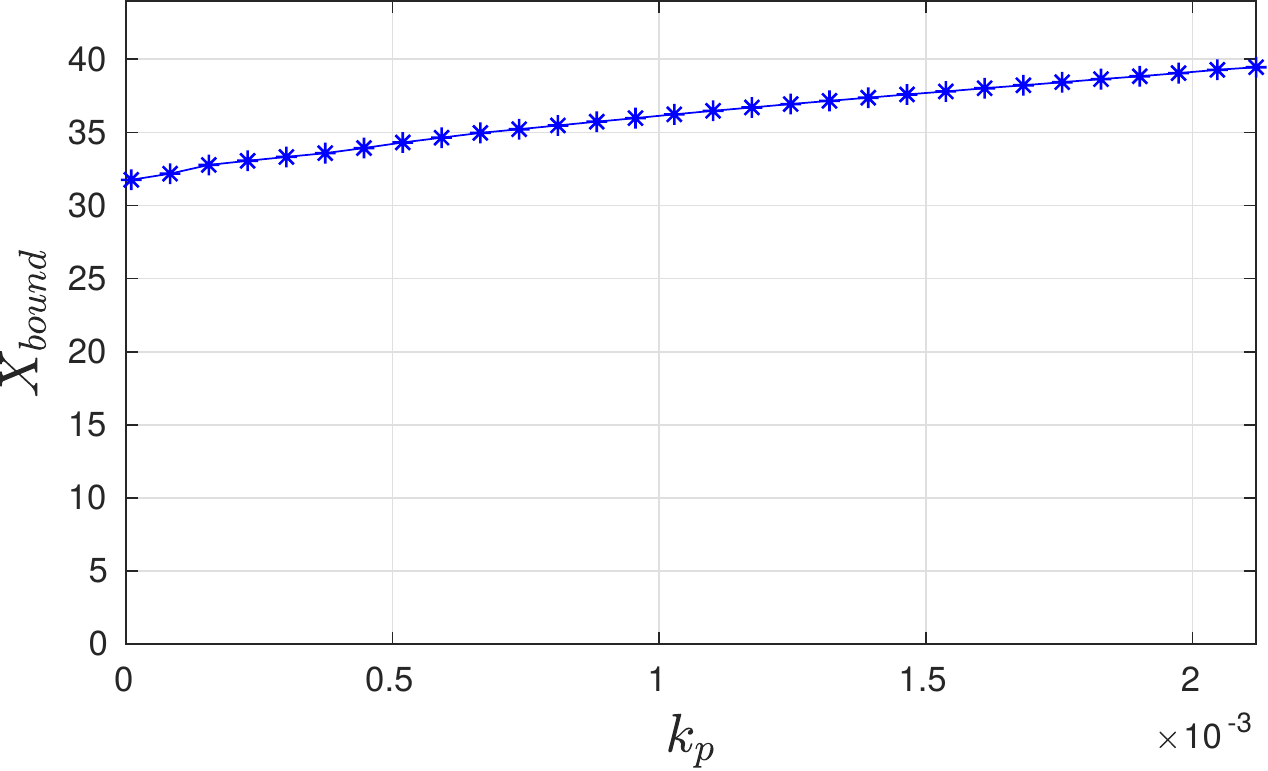} \label{fig:bestKp}}
	\caption{Minimal $X_{bound}$ obtained using Theorem~\ref{theo:practical} with $\Omega_0 = 5$ and $N = 5$.}
\end{figure} 


\section{Conclusion}

This paper focused on the analysis of the performance of a PI controller for a drilling pipe. First, a discussion between existing models in the literature allows us to conclude that the infinite-dimensional model is closer to the real drilling pipe and should then be used for simulations and analysis. Based on this last model, the exponential stability of the closed-loop system was ensured using a Lyapunov functional depending on projections of the infinite-dimensional state on a polynomials basis. This approach enables getting exponential stability of the linear system with an estimation of the decay-rate. This result was then extended to the nonlinear case considering practical stability. The example section shows that it is possible to get an estimate of the smallest attractive and invariant set and that this estimate is close to the minimal one. Further research would focus on similar analysis but using different controllers as the ones developed in \cite{canudasdewit:hal-00394990,navarro2009,hinfty} or in Chapter 10 of \cite{marquez2015analysis}.

\appendix

\subsection{Legendre Polynomials and Bessel Inequality} \label{sec:bessel}

The performance of the Lyapunov functional \eqref{eq:lyap} highly depends on the projection methodology developed in \cite{seuret:hal-01065142}. To help the reader to better understand the proof of Theorem~\ref{theo:StabLin}, a definition and some properties of Legendre polynomials is reminded. More information can be found in \cite{courant1966courant}.

\begin{definition} The orthonormal family of Legendre polynomials $\{ \mathcal{L}_k \}_{k \in \mathbb{N}}$ on $L^2([0,1])$ embedded with the canonical inner product is defined as follows:
\[
	\mathcal{L}_k(x) = (-1)^k \sum_{l = 0}^k (-1)^l \left( \begin{matrix} k \\ l \end{matrix} \right) \left( \begin{matrix} k+l \\ l \end{matrix} \right) x^l,
\]
where $\left( \begin{smallmatrix} k \\ l \end{smallmatrix} \right) = \frac{k!}{l! (k-l)!}$.
 \end{definition}

Their expended formulation is not useful in this paper but the two following lemmas are of main interest.

\begin{lemma} \label{lemma:bessel} For any function $\chi \in L^2$ and symmetric positive matrix $R \in \mathbb S^2_+$, the following Bessel-like integral inequality holds for all $N\in \mathbb N$:
	\begin{equation}\label{eq:bessel}
		\int_{0}^1 \chi^{\top}(x) R \chi(x) dx  \supeq \sum_{k=0}^{N} (2k+1)  \Chi_k^{\top} R \Chi_k
	\end{equation}
where $\Chi_k$ is the projection coefficient of $\chi$ with respect to the Legendre polynomial $\mathcal{L}_k$ as defined in \eqref{eq:chi}.
\end{lemma}

\begin{lemma} \label{lem:Chi_k}
	For any function $\chi \in L^2$ satisfying \eqref{eq:chi}, the following expression holds for any $N$ in $\mathbb N$:
	\begin{equation} \label{eq:deriv}
		\left[ \begin{smallmatrix}  \dot \Chi_0 \\ \vdots \\ \dot \Chi_{N} \end{smallmatrix} \right] = \mathbb 1_N\chi(1) - \bar {\mathbb 1}_N\chi(0) - L_N \left[ \begin{smallmatrix}  \Chi_0 \\ \vdots \\ \Chi_{N} \end{smallmatrix} \right],
	\end{equation}
	where $L_N, \mathbb{1}_N$ and $\bar{\mathbb{1}}_N$ are defined in \eqref{eq:defTheo}. 
\end{lemma}
\begin{proof} This proof is highly inspired from \cite{besselString}. Since $\chi \in L^2$ satisfies \eqref{eq:chi}, equation~\eqref{eq:chi_t} can be derived and the following holds:
\[
	\dot{\Chi}_k = \Lambda \left[ \chi(x) \mathcal{L}_k(x) \right]_0^1 - \Lambda \int_0^1\chi(x) \mathcal{L}'(x) dx - \frac{\gamma_t}{2} \left[ \begin{matrix} 1 & 1 \\ 1 & 1 \end{matrix} \right] \Chi_k.
\]
As noted in \cite{courant1966courant}, the interesting properties of Legendre polynomials are stated below:
\begin{enumerate}
	\item the boundary conditions on $\mathcal{L}_k$ ensures $\mathcal{L}_k(0) = (-1)^k$ and $\mathcal{L}_k(1) = 1$;
	\item the derivation rule for Legendre polynomials is $\frac{d}{dx} \mathcal{L}_k(x) = \sum_{j = 0}^k \ell_{j,k} \mathcal{L}_j(x)$.
\end{enumerate} 
These two properties lead to the proposed  result in equation~\eqref{eq:deriv}.
\end{proof}

\bibliographystyle{plain}
\bibliography{report_draft}

\begin{IEEEbiography}[{\includegraphics[width=1in,height=1.25in,clip,keepaspectratio]{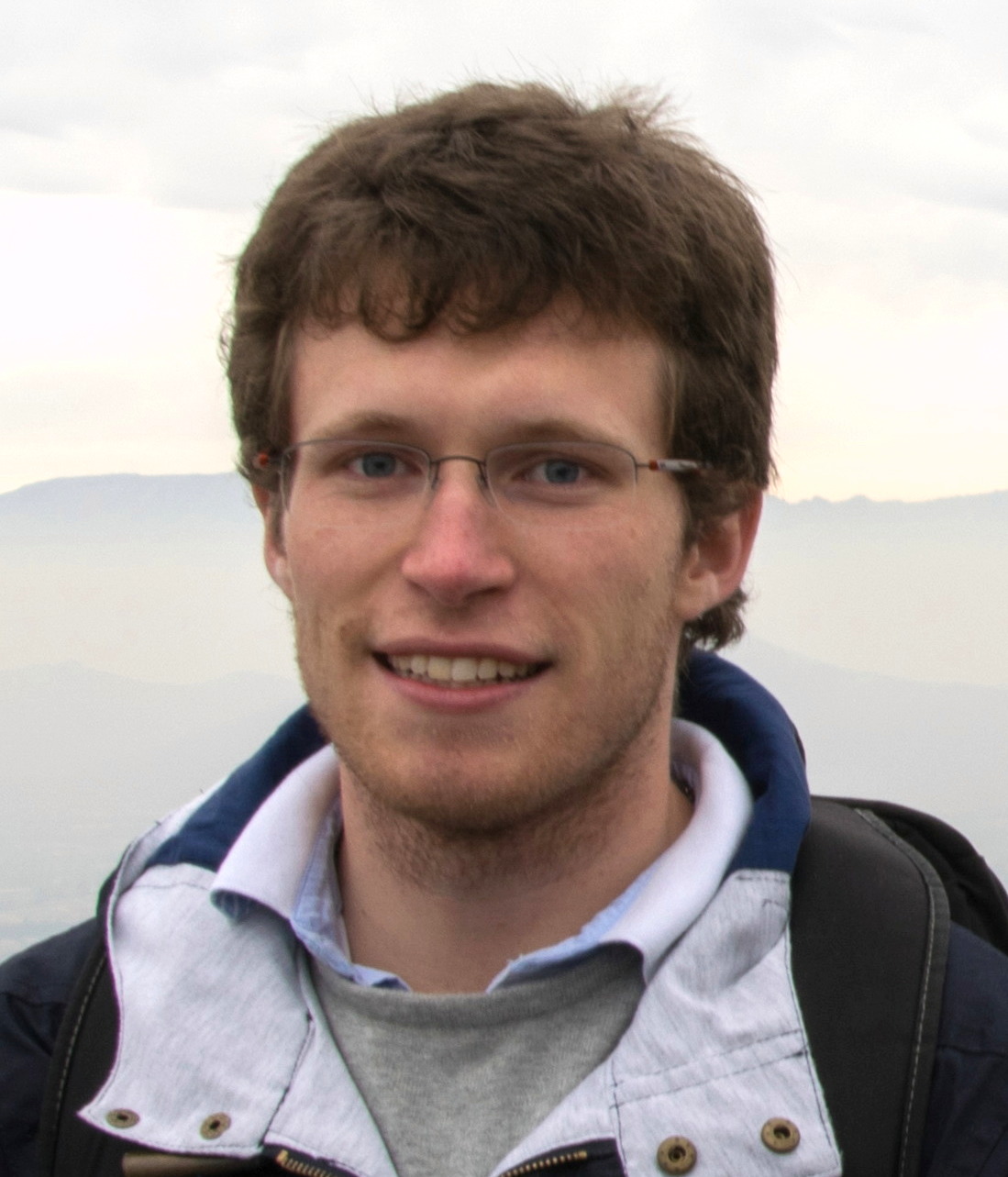}}]{M. Barreau} got the Engineer’s degree in aeronautical engineering from ISAE-ENSICA (Toulouse, France) and a double degree from KTH in Space Engineering with specialization in system engineering in 2015. He was a Ph.D. student in LAAS-CNRS `Laboratoire d'Analyse et d'Architecture des Systèmes'' in Toulouse, France under the supervision of A. Seuret and F. Gouaisbaut. He received the ``Diplôme de Doctorat'' in System Theory in 2019 from Université Paul Sabatier, Toulouse. His research interests include stability analysis of infinite dimension systems with a particular focus on drilling systems.
\end{IEEEbiography} 
\begin{IEEEbiography}[{\includegraphics[width=1in,height=1.25in,clip,keepaspectratio]{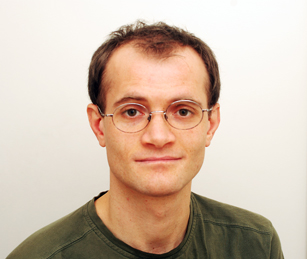}}]{F. Gouaisbaut} was born in Rennes (France) in April 26,1973. He received the ‘‘Diplôme d’Ingénieur’’ in automatic control (Engineers’ degree) from Ecole Centrale de Lille, France, in September 1997 and the ‘‘Diplôme d’Etudes Approfondies’’ (Masters’ Degree) in the same subject from the University of Science and Technology of Lille, France, in September 1997. From October 1998 to October 2001 he was a Ph.D. student at the Laboratoire d’Automatique, Génie Informatique et Signal (LAGIS) in Lille, France. He received the ‘‘Diplôme de Doctorat’’ (Ph.D. degree) from Ecole Centrale de Lille and University of Science and Technology of Lille, France, in October 2001. Since October 2003, he is an associate professor at the Paul Sabatier University (Toulouse). His research interests include time delay systems, quantized systems and robust control.
\end{IEEEbiography} 
\begin{IEEEbiography}[{\includegraphics[width=1in,height=1.25in,clip,keepaspectratio]{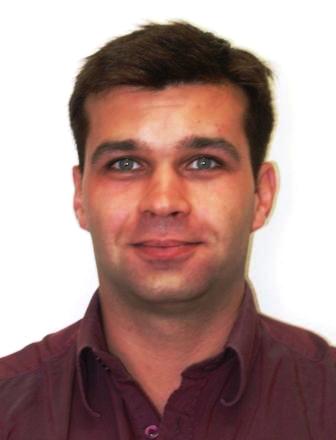}}]{A. Seuret} was born in 1980, in France. He earned the Engineer’s degree from the ‘Ecole Centrale de Lille’ (Lille, France) and the Master’s Degree in system theory from the University of Science and Technology of Lille (France) in 2003. He received the Ph.D. degree in Automatic Control from the ‘Ecole Centrale de Lille’ and the University of Science and Technology of Lille in 2006. From 2006 to 2008, he held one-year postdoctoral positions at the University of Leicester (UK) and the Royal Institute of Technology (KTH, Stockholm, Sweden). From 2008 to 2012, he was a junior CNRS researcher (Chargé de Recherche) at GIPSA-Lab in Grenoble, France. Since 2012, he has been with the ‘Laboratoire d’Architecture et d’analyse des Systèmes’’ (LAAS), in Toulouse, France as a junior CNRS researcher. His research interests include time-delay systems, networked control systems and multi-agent systems.
\end{IEEEbiography} 

\end{document}